\documentclass[sn-mathphys-num]{sn-jnl}
\usepackage{amssymb, amsmath, amsthm, mathrsfs,epic,epsfig}
\usepackage{graphicx}
\usepackage{color}
\usepackage{makecell}
\usepackage{float, caption, subcaption}

\usepackage{multirow}%
\usepackage{amsthm}%
\usepackage{mathrsfs}%

\input{epsf}

\newcommand{\bd}{\begin{description}}
\newcommand{\ed}{\end{description}}
\newcommand{\bi}{\begin{itemize}}
\newcommand{\ei}{\end{itemize}}
\newcommand{\be}{\begin{enumerate}}
\newcommand{\ee}{\end{enumerate}}
\newcommand{\beq}{\begin{equation}}
\newcommand{\eeq}{\end{equation}}
\newcommand{\beqs}{\begin{eqnarray*}}
\newcommand{\eeqs}{\end{eqnarray*}}

\newtheorem{theorem}{Theorem}[section]

\newtheorem{lemma}[theorem]{Lemma}
\newtheorem{corollary}[theorem]{Corollary}
\theoremstyle{definition}
\newtheorem{definition}[theorem]{Definition}
\newtheorem{case}{\tt Case}

\newtheorem{remark}[theorem]{Remark}

\newtheorem{question}{Question}

\begin{document}

\title{Gallai-Schur Triples and Related Problems}

\author[1]{
\fnm{Yaping}
\sur{Mao}}\email{mao-yaping-ht@ynu.ac.jp}
\author*[2]{
\fnm{Aaron}
\sur{Robertson}}\email{arobertson@colgate.edu}
\author[3]{
\fnm{Jian}
\sur{Wang}}\email{wangjian01@tyut.edu.cn}
\author[4]{
\fnm{Chenxu}
\sur{Yang}}\email{cxuyang@aliyun.com}
\author[5]{
\fnm{Gang}
\sur{Yang}}\email{gangyang98@outlook.com}

\affil[1]{\orgdiv{Faculty of Environment and Information Sciences}, \orgname{Yokohama National University}, \orgaddress{\city{Hodogaya-ku}, \state{Yokohama}, \country{Japan}}}

\affil*[2]{\orgdiv{Department of Mathematics}, \orgname{Colgate University}, \orgaddress{\city{Hamilton}, \state{New York}, \country{United States}}}

\affil[3]{\orgdiv{Department of Mathematics}, \orgname{Taiyuan University of Technology}, \orgaddress{\city{Taiyuan}, \country{China}}}

\affil[4]{\orgdiv{School of Computer Science}, \orgname{Qinghai Normal University}, \orgaddress{\city{Xining}, \state{Qinghai}, \country{China}}}

\affil[5]{\orgdiv{Graduate School of Environment and Information Sciences}, \orgname{Yokohama National University}, \orgaddress{\city{Hodogaya-ku}, \state{Yokohama}, \country{Japan}}}

\abstract{Schur's Theorem states that, for any $r \in \mathbb{Z}^+$, there
exists a minimum integer $S(r)$ such that every
$r$-coloring of $\{1,2,\dots,S(r)\}$ admits a monochromatic solution
to $x+y=z$.  Recently, Budden determined the related Gallai-Schur numbers;
that is, he determined the minimum integer $GS(r)$ such that every
$r$-coloring of $\{1,2,\dots,GS(r)\}$ admits either a rainbow or
monochromatic solution to $x+y=z$.  In this article we consider
problems that have been solved in the monochromatic setting
under a monochromatic-rainbow paradigm.  In particular, we investigate Gallai-Schur numbers when $x \neq y$,
we consider  $x+y+b=z$ and $x+y<z$, and we
investigate the asymptotic minimum number of rainbow and monochromatic
solutions to $x+y=z$ and $x+y<z$.
}


\maketitle

\section{Introduction}

A classical result in Ramsey theory is Schur's Theorem, which states that for any $r \in \mathbb{Z}^+$, there exists
a minimum integer $n$ such that every $r$-coloring of $[1,n]=\{1,2,\dots,n\}$ admits a monochromatic
solution to $x+y=z$.  Such numbers are called {\it Schur numbers} and are denoted by $S(r)$.  Only a handful of exact values
are known: $S(2)=5, S(3)=14, S(4)=45,$ and $S(5)=161$.  The largest of these was determined
in 2018 by Heule \cite{Heule}. For more details on   Schur  numbers, we refer the reader to \cite{ACPPRT,  RobertsonSchaal, RZ98, Rowley, Schoen99}.

Another classical result in Ramsey theory is the famous Ramsey Theorem \cite{Ram}.
From this theorem, it follows that every coloring of the edges
of a sufficiently large complete graph admits a complete subgraph of a given size with
the colors on all edges either the same or pairwise {distinct}.  In the latter situation we say
that the coloring is a {\it rainbow coloring}.

The numbers associated with monochromatic or rainbow substructures have been tagged with
the monicker Gallai since, in 1967, Gallai
\cite{Gallai} first examined this structure under the guise of
transitive orientations. Gallai's result was reproven in \cite{MR2063371}.

Applying this guarantee of either a monochromatic or rainbow substructure to
other Ramsey-type theorems, it is natural to investigate, in particular,
how this alters Schur's Theorem.  Since we can easily avoid rainbow structures by not using enough
colors, it is natural to require that every possible color be used.  We call such a coloring {\it exact}.

By Schur's Theorem, we may define the {\it Gallai-Schur numbers}: For $r \geq 3$, let $GS(r)$  be the minimum integer  such that
every exact $r$-coloring of $[1,GS(r)]$ admits either a monochromatic or rainbow solution to $x+y=z$.
Recently,  Budden \cite{Budden} obtained the following formula,
using results from \cite{AxenovichIverson, ChungGraham, GyarfasSarkozySeboSelkow}.

\begin{theorem}[\cite{Budden}]\label{th2-1}
For $r\geq 3$, we have
$
{GS}(r)=\begin{cases}
5^{\frac{r}{2}}+1 & \text{for $r$ even;}\\
2\cdot 5^{\frac{r-1}{2}}+1 &\text{for $r$ odd.}
\end{cases}
$
\end{theorem}

\begin{remark}\label{rem111} It is known  \cite{ACPPRT} that $\lim_{r \rightarrow \infty} \left({S}(r)\right)^{\frac{1}{r}} \geq \sqrt[5]{380} \approx 3.280626$,
while Theorem \ref{th2-1} gives us  $\lim_{r \rightarrow \infty} \left(GS(r)\right)^{\frac{1}{r}} = \sqrt{5} \approx 2.236$.
\end{remark}

 One of the main results used in \cite{Budden} concerns the Canonical Ramsey Theorem and was
 first proven (in a different context) by Chung and Graham \cite{ChungGraham}.  It is stated below as Theorem \ref{th2-1A}.
This result was also examined in more depth in both
\cite{AxenovichIverson} and \cite{GyarfasSarkozySeboSelkow}.  
Included in \cite{AxenovichIverson} is a complete characterization
of all extremal graphs for this result.

We make the following definition for use not just in Theorem \ref{th2-1A} but later in this article as well.

\begin{definition}\label{defn1}
Let $k,\ell,r \in \mathbb{Z}^+$ with $r \geq 3$.  Let $G(k,\ell;r)$ be the minimum integer 
such that every $r$-coloring of edges of the complete graph on $G(k,\ell;r)$  vertices admits either
a rainbow complete subgraph on $k$ vertices or a monochromatic complete subgraph on $\ell$ vertices.
\end{definition}
 
We now state Chung and Graham's result.

\begin{theorem}[\cite{AxenovichIverson, ChungGraham, GyarfasSarkozySeboSelkow}]\label{th2-1A}
For $r \geq 3$, we have
$$
G(3,3;r)=\begin{cases}
5^{\frac{r}{2}}+1 & \text{for $r$ even;}\\
2\cdot 5^{\frac{r-1}{2}}+1 & \text{for $r$ odd.}
\end{cases}
$$
\end{theorem}

Moving back to the topic of this article, we start by noting that
a triple $(x, y, z)$ with $x \leq y$ that satisfies $x+y=z$  is called a \emph{Schur triple}.  {Furthermore, if $x<y$, {it} is called a \emph{strict Schur triple}.}
The astute reader will notice that Schur triples have no condition of monotonicity even though Schur's name (and, hence, theorem) is
attached to them.  This is how the literature has come to reference them.  However, we will attach coloring attributes in the following definition.

\begin{definition} \label{defnGSTrip} Let $r,n \in \mathbb{Z}^+$.  Let $\chi$ be an $r$-coloring of $[1,n]$. If $x,y,z \in [1,n]$ with $x \leq y$
satisfy $x+y=z$ with either $\chi(x)=\chi(y)=\chi(z)$ or
$\chi(x),\chi(y),$ and $\chi(z)$ all distinct, then we
say $(x,y,z)$ is a {\it Gallai-Schur triple}. If $x<y$, we call it a \emph{strict Gallai-Schur triple}.
\end{definition}




The remainder of the paper is structured as follows:
in Section $2$ we give   upper
and lower bounds on strict  {Gallai-}Schur numbers;
in Section $3$ we prove that for
$b\geq 2$, letting $n(b)$ equal  $4b+10$ if $b$ is
even, and $4b+5$ if $b$ is odd, we have that every $3$-coloring of $[1,n(b)]$ admits
either a monochromatic or rainbow solution to $x+y+b=z$;
in Section $4$ we give asymptotic bounds on the minimum number of Gallai-Schur triples
over all $r$-colorings of $[1,n]$;
{in Section $5$ we give the minimum number of  monochromatic solutions to $x+y < z$
over all $3$-colorings of $[1,n]$ and investigate the minimum number of rainbow and monochromatic solutions to $x+y < z$
over all $3$-colorings of $[1,n]$}  {as well as the maximal number of rainbow solutions}; we present some open problems
in Section $6$.

\section{Strict Gallai-Schur Numbers}

\begin{definition}\label{defi-2}
Let $r$ be a positive integer. The least positive integer
$\widehat{GS}(r)$ such that every exact $r$-coloring of
$[1,\widehat{GS}(r)]$ admits a strict Gallai-Schur triple  is called a {\it strict
Gallai-Schur number}.
\end{definition}

In this section, we establish some bounds on $\widehat{GS}(r)$.  The typical
approach for finding lower bounds on Ramsey-type numbers is to find a
particular coloring that avoids the monochromatic structure.  We use
this same approach in the Gallai-Schur setting and make the following
definition for the particular type of colorings we will focus on.

\begin{definition}
Let $\chi$ be an $r$-coloring of $[1,n]$. If there is no
Gallai-Schur triple under $\chi$, we call
$\chi$ a {\it Gallai-Schur coloring} of $[1,n]$. If, in addition, $\chi(i)=\chi(n+1-i)$
for $i=1,2,\ldots,n$,  we call $\chi$ a {\it palindromic Gallai-Schur  coloring}.
\end{definition}

For an $r$-coloring $\chi$ of $[1,n]$, we shall write $\chi$ as a
string $\chi(1)\chi(2)\ldots\chi(n)$ of length $n$. Letting
$a=a_1a_2 \dots a_m$ and $b=b_1b_2 \dots b_m$, we use
$\langle a,b \rangle $ to
denote the  {concatenation of }strings $a_1a_2\ldots a_mb_1b_2\ldots b_m$.

The main result in this section is Theorem \ref{thm2.3}.  Our approach is to use
a sequence of results that build up the length of palindromic
Gallai-Schur colorings.  Theorem \ref{thm-4} then makes the connection between
palindromic
Gallai-Schur colorings
and colorings that avoid strict Gallai-Schur triples.

The first lemma we present shows that we can more than double the length of a
palindromic Gallai-Schur coloring by increasing the number of colors by 1.

\begin{lemma}\label{lem-1}
Let $n,r \in \mathbb{Z}^+$ with $r\geq 2$. If $\chi$ is a
palindromic Gallai-Schur  $r$-coloring of $[1,n]$, then $\chi^{*}=\langle \chi, r+1, \chi \rangle$ is a
palindromic Gallai-Schur  $(r+1)$-coloring of $[1,2n+1]$.
\end{lemma}

\begin{proof}
Clearly, $\chi^{*}$ is a  {palindromic coloring}. Suppose to the contrary that
there is a Gallai-Schur  triple $(x,y,z)$ with
$x\leq y<z$ under $\chi^{*}$. Since $\chi$ is a
Gallai-Schur coloring, we have $z\geq n+1$. If $z=n+1$,
since $n+1$ is the only integer of
color $r+1$, then $(x,y,z)$ has to be rainbow.
Note that $x\leq
y\leq n$ and $y=n+1-x$.
Since $\chi$ is  {a palindromic coloring}, it follows that
 {$\chi^{*}(x)=\chi(x)=\chi(n+1-x)=\chi^*(n+1-x)
=\chi^{*}(y)\neq \chi^{*}(z)$}, a
contradiction. Thus $z\geq n+2$. If $x\geq n+1$,
then $z=x+y\geq
2n+2$, a contradiction. Thus $x\leq n$.

We finish the proof by considering 3 cases that exhaust all possibilities.

\setcounter{case}{0}
\begin{case}
$y\geq n+2$.
Let $y'=y-n-1$ and $z'=z-n-1$. From Remark \ref{rem2.2}, we have $\chi^*(i)=\chi^*(i+n+1)$ for
$i\in [1,n]$, so that $\chi^*(y')=\chi^*(y)$ and
$\chi^*(z')=\chi^*(z)$. Then $(x,y',z')$ or $(y',x,z')$ forms a rainbow or
monochromatic Schur triple under $\chi$, which contradicts the fact
that $\chi$ is a Gallai-Schur $r$-coloring of $[1,n]$.
\end{case}

\begin{case}
$y= n+1$.
Since $\chi^*(i)=\chi^*(i+n+1)$ for $i\in [1,n]$, it follows
that $\chi^*(x)=\chi^*(x+y)=\chi^*(z)$. However, $y$ is the only
integer of color $r+1$. Hence, $\chi^*(x)=\chi^*(z)\neq
\chi^*(y)$, a contradiction.
\end{case}

\begin{case}
$y\leq n$.
Let $x'=n+1-x$ and $z'=z-n-1$, {so that $x', z'\in[1,n]$}. Since $\chi$ is  {a palindromic coloring}, it
follows that $\chi^*(x') =\chi^*(x)$.  Since
$\chi^*(i)=\chi^*(i+n+1)$ for $i=1,2,\ldots,n$, it follows that
$\chi^*(z')=\chi^*(z)$. Moreover, we have $x'+z'=z-x=y$. Thus
$(x',z',y)$ or $(z',x',y)$ forms a Gallai-Schur  triple under
$\chi$, which contradicts {the fact that $\chi$} is a Gallai-Schur
$r$-coloring of $[1,n]$.
\end{case}
\end{proof}

The next   lemma is used  only  in furtherance of this section's main result.
As such, we state the lemma  but place the proof (which is similar to
the proof of Lemma \ref{lem-1}) in the Appendix.

\begin{lemma}\label{lem-2}
Let $n,r \in \mathbb{Z}^+$ with $r\geq 2$. If $\chi$ is a
palindromic Gallai-Schur  $r$-coloring of $[1,n]$, then $\chi^{**}= \langle \chi, r+1, \chi, r+2,
\chi, r+2, \chi, r+1, \chi \rangle$ is
a palindromic Gallai-Schur  $(r+2)$-coloring of $[1,5n+4]$.
\end{lemma}

Using Lemma \ref{lem-2}, the next result easily follows.

\begin{theorem}\label{thm-3}
For every $k$, there is a palindromic Gallai-Schur  $2k$-coloring of
$[1,5^{k}-1]$ and  a palindromic Gallai-Schur  $(2k+1)$-coloring of
$[1,2\cdot 5^{k}-1]$.
\end{theorem}

\begin{proof}
We prove the theorem by induction on $k$. For $k=1$, it is easy to
see that $1221$ is a palindromic Gallai-Schur
 $2$-coloring of $[1,4]$
and $122131221$ is a palindromic Gallai-Schur
$3$-coloring of $[1,9]$. Assume $\chi$ is a palindromic Gallai-Schur  $(2k-2)$-coloring of $[1,5^{k-1}-1]$. We will prove that there exists a palindromic Gallai-Schur  $2k$-coloring of $[1,5^{k}-1]$ and a palindromic Gallai-Schur  $(2k+1)$-coloring of $[1,2\cdot 5^{k}-1]$.  {By
Lemma \ref{lem-2}, $\chi^{**}$ is a Gallai-Schur
palindromic $2k$-coloring of $[1,5^{k}-1]$.
Then it follows from Lemma \ref{lem-1} that $(\chi^{**})^*$ is a
palindromic Gallai-Schur  $(2k+1)$-coloring of $[1,2\cdot 5^{k}-1]$.}
Thus
the theorem follows.
\end{proof}

 {Although Budden's result in Theorem \ref{th2-1} provides us with the
 exact values of the Gallai-Schur numbers,   Theorem \ref{thm-3} allows us to provide a useful (weaker)
 lower bound for
the Gallai-Schur numbers. The usefulness is with the relationship between
palindromic colorings and strict Gallai-Schur numbers.
We will refer to a coloring with no strict Gallai-Schur triple as 
a {\it strict Gallai-Schur coloring}.

\begin{theorem}\label{thm-4}
Let $n,r \in \mathbb{Z}^+$ with $r\geq 3$. If $\chi$ is a palindromic Gallai-Schur
 $r$-coloring of $[1,n]$, then
 \[
\chi^{+}= \langle \chi, r+1, \chi, r+1, \chi, r+2, \chi, r+1, \chi, r+2,
\chi, r+2,\chi, r+2, \chi, r+1, \chi\rangle
\]
 is a strict
Gallai-Schur $(r+2)$-coloring of $[1,9n+8]$.
Hence, 
\end{theorem}

\begin{proof}
Suppose, for a contradiction, that there is a Gallai-Schur triple
$(x,y,z)$ with $x<y<z$ under $\chi^+$.  {Consider such a triple }$(x,y,z)$  {with minimal $z$}.
Let $A=\{n+1, 2n+2, 3n+3, 4n+4, 5n+5, 6n+6, 7n+7, 8n+8\}$. The
coloring  {string} of $A$ under $\chi^+$ is
$r+1,r+1,r+2,r+1,r+2,r+2,r+2,r+1$. It is easy to see that
there is no  monochromatic strict Schur triple in $A$ under
$\chi^+$. Note that if two of $x,y,z$ are in $A$ then all of them
must be in $A$,  {but then $(x,y,z)$ is not a strict Gallai-Schur triple.}
Thus at least two of $x,y,z$ are not in $A$. If $z\in
A$,  {that is $z=in+i$ for some $i\in [1,8]$, then $\chi^+(z)\geq r+1$}. From $x,y\notin A$ we infer that $\chi^+(x) =\chi^+(in+i-x)
 {=\chi^+(n+1-x)}
=\chi^+(y)<r+1$, which
contradicts the assumption that $(x,y,z)$
is  {a Gallai-Schur triple with $x<y<z$}.  {Hence, $z \not\in A$.} If
$x\in A$; that is, $x=in+i$ for some
$i\in [1,8]$, then {$\chi^+(x)\geq r+1$ and  {(by construction)}}
$\chi^+(y)=\chi^+(y+in+i)=\chi^+(z)<r+1$,
a contradiction.  {Hence, $x \not\in A$.  By an identical argument we obtain $y \not\in A$.}
Hence we may assume that none of $x,y,z$ is in $A$.

Let $I_i=[in+i+1,(i+1)n+i]$ for $i=0,1,2,
\ldots,8$.  Clearly, $x,y,z\in \bigcup_{i=0}^8
I_i$. First we show that $x\in I_0$.
Otherwise, setting $x'=x-n-1$
 and $z'=z-n-1$,
it is easy to see that $\chi^+(x')
=\chi^+(x)$, $\chi^+(z')=\chi^+(z)$,
$x'<z'$, and $x'+y=z'$ so that $(x',y,z')$
is a Gallai-Schur triple with $x'<y<z'$
under $\chi^+$, which contradicts
the minimality of $z$. Thus, $x\in I_0$.

Next we show that $y$ and $z$ must be in the different  $I_i$'s. Suppose,  {for a contradiction,} that there exists $i$ such that $y,z\in I_i$. Since $\chi$ is  {a Gallai-Schur coloring}, it follows that $i\geq 1$. Let $y'=y-i(n+1)$ and $z'=z-i(n+1)$. Then $\chi^+(y')=\chi^+(y)$, $\chi^+(z')=\chi^+(z)$ and $x+y'=z'$ so that $(x,y',z')$ is a Gallai-Schur triple under $\chi$, again a contradiction to the minimality of $z$. Thus, $y$ and $z$ are in the different intervals.

If both $x$ and $y$ are in $I_0$, then $z\in I_1$. Let $x'=n+1-x$ and $z'=z-n-1$. Since $\chi$ is {a palindromic coloring}, we have $\chi^+(x')=\chi^+(x)$. Moreover, $\chi^+(z')=\chi^+(z)$ and $x'+z'=z-x=y$. Hence, $(x',z',y)$ is a Gallai-Schur triple under $\chi$, contradicting the fact that $\chi$ is
{a Gallai-Schur coloring}. Thus $y \not \in I_0$.
We are left with the case $x\in I_0$, $y\in I_i$ with $i\geq 1$ and $z\in I_{i+1}$. Let $x'=n+1-x$, $y'=y-i(n+1)$ and $z'=z-(i+1)(n+1)$. Similarly, $\chi^+(x')= \chi^+(x)$, $\chi^+(y')=\chi^+(y)$, $\chi^+(z')=\chi^+(z)$ and $x'+z'=z-x-i(n+1)=y'$ implying that $(x',z',y')$ is a Gallai-Schur triple under $\chi$, which contradicts the fact that $\chi$ is {a Gallai-Schur coloring}. 

 {We may now} conclude that $\chi^+$ is a strict Gallai-Schur $(r+2)$-coloring of $[1,9n+8]$.
\end{proof}

Theorem \ref{thm-4}, together with Theorem \ref{th2-1},
will provide us with a lower bound for $ \widehat{GS}(r)$, while the next lemma will give
us an upper bound.

\begin{lemma}\label{lem2.5} Let $\{a_i\}^n_{i=1}$ be an increasing sequence of non-negative integers
with no $3$-term arithmetic progression with $n \geq GS(r)$.  Then
every $r$-coloring of  {$[1,a_n]$} contains a
strict Gallai-Schur triple.
\end{lemma}

\begin{proof}
We will show that there exist $a_i<a_j<a_k$ such that
 { {$(a_k-a_j, a_j-a_i, a_k-a_i)$ with $a_k-a_j \neq  a_j-a_i$  is a  strict Gallai-Schur triple under any $r$-coloring $\chi$ of $[1,a_n]$.}} 
 To this end, label the vertices of the complete graph $K_n$ by
$a_1,a_2,\ldots, a_n$.  {For any $r$-coloring $\chi$ of $[1,a_n]$, we color each edge $a_ia_j$ of $K_n$  by the color $\chi(|a_j-a_i|)$.}

 Since $n \geq GS(r)$, recalling Definition \ref{defn1},  it follows from Theorems \ref{th2-1A} and \ref{th2-1} that $n \geq G(3,3;r)=GS(r)$
so that $K_n$  contains a rainbow or
monochromatic triangle.  Let  {$(a_i,a_j,a_k)$} be such a triangle {with $a_i < a_j <a_k$}.  Since $a_k-a_i =
(a_k-a_j)+(a_j-a_i)$ and
$a_k-a_j \neq a_j-a_i$ (because $a_i, a_j,
a_k$ does not form a $3$-term arithmetic progression), by setting $x=a_k-a_j$, $y=a_j-a_i$ and $z=a_k-a_i$, we see that $(x,y,z)$ is the desired triple  {under $\chi$, an arbitrary   $r$-coloring  of $[1,a_n]$}.
\end{proof}

\begin{theorem}\label{thm2.3}
For $r\geq 5$, we have
$
f(r) \leq \widehat{GS}(r)<\frac{43}{4}\cdot\left(3^{\log_2 \sqrt{5}}\right)^{r},
$
where
$$
f(r)=\begin{cases}
\frac{9}{5} \cdot 5^{\frac{r}{2}}+9 & \text{for $r$ even;}\\
\frac{18}{5}\cdot 5^{\frac{r-1}{2}}+9& \text{for $r$ odd.}
\end{cases}
$$
\end{theorem}

\begin{proof}
By Theorems  {\ref{th2-1},} \ref{thm-3}, and \ref{thm-4}, we obtain $\widehat{GS}(r)-1 \geq 9(GS(r-2)-1)+8$
so that
$$\widehat{GS}(r) \geq 9GS(r-2) = 
\begin{cases}
\frac{9}{5} \cdot 5^{\frac{r}{2}}+9 & \text{for $r$ even;}\\
\frac{18}{5}\cdot 5^{\frac{r-1}{2}}+9& \text{for $r$ odd.}
\end{cases}
$$ Thus, we are left to prove the upper bound for
$\widehat{GS}(r)$.

Let $n=GS(r)$ and let $\{a_i\}^n_{i=1}$ be the following  increasing sequence of positive integers
with no $3$-term arithmetic progression ({which is well-known as a Stanley sequence}):
{$a_1< \cdots <a_n$ are the first $n$ non-negative integers whose ternary representations have only the digits $0$ and $1$}. Clearly, we have
$a_n\leq 3^{\log_2 n+1}$. By Lemma \ref{lem2.5} we have
$$
\widehat{GS}(r)\leq a_n\leq 3^{\log_2 GS(r)+1}= \begin{cases}
3^{\log_2(2\cdot 5^{r/2}+1)+1} & \text{for $r$ even;}\\
3^{\log_2(4\cdot 5^{(r-1)/2}+1)+1} & \text{for $r$ odd.}
\end{cases}
$$
Moreover,
$
3^{\log_2(2\cdot 5^{r/2}+1)+1} \leq 3^{\log_2 5^{(r+1)/2}+1} =3\cdot\left(3^{\log_2 \sqrt{5}}\right)^{r+1}< \frac{43}{4}\cdot\left(3^{\log_2 \sqrt{5}}\right)^{r}
$
and 
$
3^{\log_2(4\cdot 5^{(r-1)/2}+1)+1} \leq 3\cdot 3^{\log_2 5^{(r+1)/2}} < \frac{43}{4}\cdot\left(3^{\log_2 \sqrt{5}}\right)^{r},
$
proving the theorem.
\end{proof}

\begin{remark} Noting that $3^{\log_2 \sqrt{5}} < \sqrt[5]{380}$ and comparing the upper bound in Theorem \ref{thm2.3} with the lower bound for
the growth rate of the Schur numbers found in Remark \ref{rem111}, we find that the growth rate
of $\widehat{GS}(r)$ is strictly smaller than that of $S(r)$.
\end{remark}

\section{The Equation $x+y+b=z$}

In this section, we  give   exact values   for some numbers closely
linked with Gallai-Schur numbers. In particular, we consider
monochromatic and rainbow solutions to $x+y+b=z$ with $b \in \mathbb{Z}^+$.
We start by stating the result obtained by Schaal for
the monochromatic situation with 2 and 3 colors.
The $2$-color case of the following theorem is found in \cite{Schaal1},
while the $3$-color case is in \cite{Schaal2}.  

\begin{theorem}[\cite{Schaal1}, \cite{Schaal2}]\label{th-LR04}
Let $b$ be a positive integer. The minimum integer $m(b)$ such that every $2$-coloring of $[1,m(b)]$
admits a monochromatic solution to $x+y+b=z$ is $m(b)=4b+5$.
The minimum integer $m'(b)$ such that every $3$-coloring of $[1,m'(b)]$
admits a monochromatic solution to $x+y+b=z$ is $m'(b)=13b+4$.
\end{theorem}

Investigating the monochromatic-rainbow paradigm for this new equation,
we will use the following notation.

\vskip 5pt
\noindent
{\bf Notation.} Let $b \in \mathbb{Z}^+$.  Let $n(b)$ be the minimum integer such that
every exact $3$-coloring of $[1,n(b)]$ admits either a monochromatic or rainbow solution to
$x+y+b=z$.

\vskip 5pt
We start with a lower bound for $n(b)$.

\begin{lemma}\label{GRLowerBound}
For $k\in \mathbb{Z}^+$, we have
$
n(2k)\geq 8k+10$
and
$n(2k+1)\geq 8k+9.
$
\end{lemma}
\begin{proof}
We start by showing that $n(2k)\geq 8k+10$.
For $i\in [8k+9]$, define
\[
\chi(i)=\begin{cases}
1 & \text{ if $i$ is odd;}\\
3 & \text{ if $i$ is even and $i\in [2k+4,6k+6]$;}\\
2 & \text{ otherwise.}
\end{cases}
\]
We shall show that $\chi$ is a 3-coloring of $[1,8k+9]$ with neither a
monochromatic nor rainbow solution to $x+y+2k=z$.  {For a contradiction,} suppose not, and let
$(x,y,z)$ be such a solution  with $x\leq y< z$.
If $x$ and $y$ are both odd, then $z$ is even. Since $\chi(i)=1$ for odd
$i$ and $\chi(i)\neq 1$ for even $i$, it follows that $(x,y,z)$ cannot be rainbow
or monochromatic, a contradiction. For the same reason, it cannot
happen that exactly one of $x$ and $y$ is odd. Thus,  all of $x,y,z$ are
even.

Let $I_1=[2,2k+2], I_2=[2k+4,6k+6]$ and $I_3=[6k+8,8k+8]$. It is
easy to see that even integers in $I_1\cup I_3$ are colored 2 under
$\chi$ and even integers in $I_2$ are colored $3$ under $\chi$. If
$x,y\in I_1$, then $z=x+y+2k\in I_2$. It follows that
$\chi(x)=\chi(y)=2$ and $\chi(z)=3$, a contradiction. If  $x,y\in
I_2$, then $z=x+y+2k\in I_3$. It follows that $\chi(x)=\chi(y)=3$
and $\chi(z)=2$, a contradiction. If $x,y\in I_3$, then
$z=x+y+2k>8k+8$, a contradiction. Thus $x,y$ cannot fall in the same
interval. If $y\in I_3$, then $z=x+y+2k>8k+8$, a contradiction. It
follows that   $x\in I_1$ and $y\in
I_2$. But then $(x,y,z)$ has to be a rainbow solution, which
contradicts the fact that $\chi(i)=2$ or 3 for even $i$. Therefore,
$\chi$ is a 3-coloring of $[1,8k+9]$ avoiding monochromatic and rainbow
solutions to $x+y+2k=z$, thereby showing that $n(2k)\geq
8k+10$.

We next show that $n(2k+1)\geq 8k+9$. For $i\in [1,8k+8]$, define
\[
\chi(i)=\begin{cases}
2 & \text{ if $i$ is even,}\\
3 & \text{ if $i$ is odd and $i\in [2k+3,6k+5]$,}\\
1 & \text{ otherwise.}
\end{cases}
\]
We shall show that $\chi$ is a $3$-coloring of $[1, 8k+8]$ without
monochromatic and rainbow solutions to $x+y+2k+1=z$. Suppose otherwise and let
$(x,y,z)$ be such a solution with $x\leq y< z$.
If $x$ and $y$ are both even, then $z$ is odd. Since $\chi(i)=2$ for even
$i$ and $\chi(i)\neq 2$ for odd $i$, it follows that $(x,y,z)$ cannot be rainbow and
monochromatic, a contradiction. For the same reason, it cannot
happen that one of $x$ and $y$ is even. Thus, we have that all of $x,y,z$ are
odd.

Let $I_1=[1,2k+1], I_2=[2k+3,6k+5]$ and $I_3=[6k+7,8k+7]$. It is
easy to see that odd integers in $I_1\cup I_3$ are colored $1$ under
$\chi$ and odd integers in $I_2$ are colored $3$ under $\chi$. If
$x,y\in I_1$, then $z=x+y+2k+1\in I_2$. It follows that
$\chi(x)=\chi(y)=1$ and $\chi(z)=3$, a contradiction. If  $x,y\in
I_2$, then $z=x+y+2k+1\in I_3$. It follows that $\chi(x)=\chi(y)=3$
and $\chi(z)=1$, a contradiction. If $x,y\in I_3$, then
$z=x+y+2k+1>8k+7$, a contradiction. Thus $x$ and $y$ cannot fall in the same
interval. If $y\in I_3$, then $z=x+y+2k+1>8k+7$, a contradiction. It
follows that   $x\in I_1$ and $y\in
I_2$. But then $(x,y,z)$ must be a rainbow solution, which
contradicts the fact that $\chi(i)=1$ or $3$ for odd $i$. Therefore,
$\chi$ is a $3$-coloring of $[1,8k+8]$ that admits no monochromatic nor
rainbow solution to $x+y+2k+1=z$.  Hence,
$n(2k+1)
\geq 8k+9$.
\end{proof}

\begin{theorem}\label{th3-color}
For $k \in \mathbb{Z}^+$, we have
$
n(2k)=8k+10$
and
$
n(2k+1)=8k+9.
$
Furthermore, $n(1)=11$.
\end{theorem}

\begin{proof} It is easy to check by hand that $n(1)=11$ (note that $\{1,4,7,10\}, \{2,9\},\break \{3,5,6,8\}$
gives a $3$-coloring with neither a monochromatic nor rainbow solution to $x+y+1=z$).
We now consider $n(2k)$ and $n(2k+1)$ for $k \in \mathbb{Z}^+$.
By Lemma \ref{GRLowerBound}, it suffices to show that
every (exact) 3-coloring of $[1,8k+10]$ admits either a monochromatic or rainbow solution to $x+y+2k=z$
and that every (exact) 3-coloring of $[1,8k+9]$ admits either a monochromatic or rainbow solution
to $x+y+2k+1=z$.  To do so, we turn to the Maple program {\tt GALRAD}, written by the
second author\footnote{available at {\tt http://math.colgate.edu/$\sim$aaron}}.

{\tt GALRAD} automates the coloring of integers (within either $[1,8k+10]$ or $[1,8k+9]$) by
considering forced colors when avoiding both monochromatic and rainbow solutions.  Letting $A, B,$ and $C$
be the color classes, we may assume that $1 \in A$ and, consequently, we may assume that $b+2 \in B$,  {where $b=2k$ or $2k+1$}.
From this, we must have either $2b+3, 4b+5 \in A$ and $3b+4 \in C$ or $2b+3 \in B$, $3b+4 \in A$, and
$4b+5 \in C$.  This is the starting point of {\tt GALRAD}.  We then may input additional integers
of assumed colors.  We either obtain a contradiction (e.g., $x$ must be both in $A$ and not in $A$) or
we obtain a (typically, larger) list of integers with forced colors and continue with assumptions.

\begin{table}[h!]\centering\footnotesize \hspace*{0pt}
 \begin{subtable}[t]{.4\textwidth}\setlength{\tabcolsep}{.5em}
\begin{tabular}[t]{c|c|c|c}
$A$&$B$&$C$&Contra-\\
&&&diction\\\hline\hline
3&&& No\\[-.5pt]
3&&$k$& Yes\\[-.5pt]
3&$k$&& Yes\\[-.5pt]
$3,k$&&& No\\[-.5pt]
$3,k$&2&& Yes\\[-.5pt]
$3,k$&&2& Yes\\[-.5pt]
$2,3,k$&&& No\\[-.5pt]
$2,3,k,2k+1$&&& Yes\\[-.5pt]
$2,3,k$&$2k+1$&& Yes\\[-.5pt]
$2,3,k$&&$2k+1$& Yes\\[-.5pt] \hline
&3&& No\\[-.5pt]
2&3&& Yes\\[-.5pt]
&$2,3$&& Yes\\[-.5pt]
&3&2& No\\[-.5pt]
$2k-1$&3&2& Yes\\[-.5pt]
&$3,2k-1$&2& Yes\\[-.5pt]
&3&$2,2k-1$& Yes\\[-.5pt] \hline
&&3& No\\[-.5pt]
2&&3& Yes\\[-.5pt]
&&2,3& Yes\\[-.5pt]
&2&3& No\\[-.5pt]
$2k+1$&2&3& Yes\\[-.5pt]
&$2,2k+1$&3& Yes\\[-.5pt]
&2&$3,2k+1$& Yes\\[-5pt]
\end{tabular}
\vspace*{5pt}
\subcaption{\footnotesize Forced colors with $b=2k+1$  and $1,2b+3, 4b+5 \in A$ and $3b+4 \in C$}
\end{subtable} \hspace*{30pt}
\begin{subtable}[t]{.4\textwidth}\footnotesize\setlength{\tabcolsep}{.5em}
\begin{tabular}[t]{c|c|c|c}
$A$&$B$&$C$&Contra-\\
&&&diction\\\hline\hline
 $2k+1$&&& No\\[-.5pt]
  $2,2k+1$& && Yes\\[-.5pt]
 $2k+1$&2&& Yes\\[-.5pt]
 $2k+1$& &2& Yes\\[-.5pt] \hline
 &$2k+1$&& No\\[-.5pt]
 &$2,2k+1$&& Yes\\[-.5pt]
 &$2k+1$&2& Yes\\[-.5pt]
$2$&$2k+1$&& No\\[-.5pt]
$2,k$&$2k+1$&& Yes\\[-.5pt]
$2$&$k,2k+1$&& Yes\\[-.5pt]
$2$&$2k+1$&$k$& Yes\\[-.5pt] \hline
&&$2k+1$& No\\[-.5pt]
&2&$2k+1$& Yes\\[-.5pt]
&&$2,2k+1$& Yes\\[-.5pt]
2&&$2k+1$& No\\[-.5pt]
$2,k$& &$2k+1$& Yes\\[-.5pt]
2&$k$&$2k+1$& Yes\\[-.5pt]
2& &$2,2k+1$& Yes\\[-5pt]
\end{tabular}
\vspace*{5pt}
\subcaption{\footnotesize Forced colors with $b=2k$  and $1,2b+3, 4b+5 \in A$ and $3b+4 \in C$}
\end{subtable}
\vskip10pt
\hspace*{0pt}
 \begin{subtable}[t]{.4\textwidth}\setlength{\tabcolsep}{.35em}
\begin{tabular}[t]{c|c|c|c}
$A$&$B$&$C$&Contra-\\
&&&diction\\\hline\hline
$2k+1$&&& No\\[-.5pt]
$3,2k+1$&&$k$& Yes\\[-.5pt]
$2k+1$&3&& Yes\\[-.5pt]
$2k+1$&&3& Yes\\[-.5pt] \hline
&&$2k+1$& No\\[-.5pt]
$3$&&$2k+1$& Yes\\[-.5pt]
 &3&$2k+1$& Yes\\[-.5pt]
 &&$3,2k+1$& Yes\\[-.5pt] \hline
 &$2k+1$&& No\\[-.5pt]
 &$3,2k+1$&& Yes\\[-.5pt]
 &$2k+1$&3& Yes\\[-.5pt]
$3$&$2k+1$&& No\\[-.5pt]
$3,k-1$&$2k+1$&& Yes\\[-.5pt]
$3$&$k-1,2k+1$&& Yes\\[-.5pt]
$3$&$2k+1$&$k-1$& No\\[-.5pt]
$3,k+2$&$2k+1$&$k-1$& Yes\\[-.5pt]
$3$&$k+2,2k+1$&$k-1$& Yes\\[-.5pt]
3&$2k+1$&$k-1,k+2$& No\\[-5pt]
\end{tabular}
\vspace*{5pt}
\subcaption{\footnotesize Forced colors with $b=2k+1$  and $1, 3b+4 \in A$, $2b+3\in B$, and $4b+5 \in C$}
\end{subtable} \hspace*{30pt}
\begin{subtable}[t]{.4\textwidth}\footnotesize\setlength{\tabcolsep}{.25em}
\begin{tabular}[t]{c|c|c|c}
$A$&$B$&$C$&Contra-\\
&&&diction\\\hline\hline
 $2k+1$&&& No\\[-.5pt]
  $2,2k+1$& && Yes\\[-.5pt]
 $2k+1$&2&& Yes\\[-.5pt]
 $2k+1$& &2& Yes\\[-.5pt] \hline
&&$2k+1$& No\\[-.5pt]
2&&$2k+1$& Yes\\[-.5pt]
&2&$2k+1$& Yes\\[-.5pt]
&&$2,2k+1$& Yes\\[-.5pt] \hline
&{$2k+1$}&& {No}\\[-.5pt]
$k+1$& $2k+1$&& Yes\\[-.5pt]
& $k+1,2k+1$&& Yes\\[-.5pt]
 &$2k+1$&$k+1$& No\\[-.5pt]
 &$3,2k+1$&$k+1$& Yes\\[-.5pt]
 &$2k+1$&$3,k+1$& Yes\\[-.5pt]
$3$&$2k+1$&$k+1$& No\\[-5pt]
\end{tabular}
\vspace*{5pt}
\subcaption{\footnotesize Forced colors with $b=2k$  and $1, 3b+4 \in A$, $2b+3\in B$, and $4b+5 \in C$ }
\end{subtable}
\vspace*{5pt}
\caption{\small Forced colors using {\tt GALRAD}}\label{tabGR1}
\end{table}

In Table \ref{tabGR1}, underneath the color classes ($A, B,$ and $C$) we place integers that
we assume are in that color class.  The last column informs us of whether or not we obtain a
contradiction.  Parts (a) and (b) each exhibit a complete list of possibilities, thereby proving the upper bounds.
 We see that both (c) and (d) of Table 1 end without a final contradiction.  To finish both of these cases some
 additional work is needed.  For the situation in (c), we use the flexibility offered in {\tt GALRAD} to further
 assume that $k \geq k_0$, where $k_0$ is a given integer. This allows us to consider $k-1, k-2, \dots, k-k_0+1$ as
 positive integers in the computer algorithm.  By iteratively increasing $k_0$ by 1
it becomes clear
 (but is tedious to show by hand) that when we are in the final situation of part (c) we have
 $[1,4k_0+5] \subseteq A$.  Taking $k=k_0$, we obtain $[1,4k+5] \subseteq A$, which contradicts
 the deduction that $k-1 \in C$.  For situation (d), we argue in
 the same way noting that $[1,4k+8] \subseteq A$ while $k+1 \in C$.
\end{proof}

\section{On the Number of Gallai-Schur Triples}

In 1995, Graham, R\"{o}dl, and Ruci\'{n}ski \cite{GRR96} proposed
the following multiplicity problem: Find (asymptotically) the least number of monochromatic solutions to $x+y=z$ that must occur in a $2$-coloring of the set $[1,n]$. This problem was solved by Robertson and Zeilberger \cite{RZ98}, and independently by Schoen \cite{Schoen99}, with a nice proof given later by Datskovsky \cite{Datskovsky03}. The answer was found to be $\frac{n^2}{22}(1+o(1))$.

Here we investigate this problem for Gallai-Schur triples.
We start with the following lemma that
uses notation from Definition \ref{defn1}.

\begin{lemma}\label{lem4.1new} Let $k,\ell,r \in \mathbb{Z}^+$ and let $n\geq G(k,\ell;r)$.
Define $s=s(k,\ell;r)$ to be the minimum number of rainbow $K_k$ and monochromatic $K_\ell$ copies over
all $r$-colorings of the edges of the complete graph on ${G(k,\ell;r)}$ vertices.  Then any
$r$-coloring of the edges of $K_n$ contains at least
$$
s \frac{{n \choose m}}{ {G(k,\ell;r) \choose m}}
$$
rainbow $K_k$ copies and monochromatic $K_\ell$ copies, where $m=\min(k,\ell)$.
\end{lemma}

\begin{proof}
  Denote the vertex set of a graph $H$ by $V(H)$, and
let $K(W)$ be the complete graph on vertex set $W$.
Let $g=G(k,\ell;r)$, and let $\chi$ be an
$r$-coloring of the edges of $K_n$.

Call   $T \subseteq V(K_n)$   {\it good} if either
$|T|=k$ and $K(T)$ is rainbow or $ |T|=\ell$ and $K(T)$
is monochromatic.
Define
\[
\Omega = \left\{(S,T)\colon T\subseteq S\subseteq V(K_n),\  |S|=g,\ T \mbox{ is good}\right\}.
\]

We shall derive a lower bound for the number of rainbow and
monochromatic triangles under $\chi$ by a double counting technique.
On one hand, for any $S\subseteq V(K_n)$ with $|S|=g$,  {it is clear that} $K(S)$ contains at least $s$ rainbow $K_k$ and  monochromatic $K_\ell$ copies
under $\chi$. It follows that  {$|\Omega|\geq s\binom{n}{g}$}. On the
other hand, every rainbow $K_k$  is contained in
$\binom{n-k}{g-k}$ different subgraphs $K_g$ of $K_n$, and every monochromatic $K_\ell$
is contained in $\binom{n-\ell}{g-\ell}$ different subgraphs $K_g$ of $K_n$. Let $t_\chi$ be the total
number of rainbow $K_k$ and monochromatic $K_\ell$ copies under $\chi$. Then we have
\[
s {n \choose g} \leq |\Omega| \leq t_\chi \cdot \max\left(\binom{n-k}{g-k},\binom{n-\ell}{g-\ell}\right)
\]
and hence
\[
t_\chi \geq s \frac{\binom{n}{g}}{\binom{n-m}{g-m} }= s \frac{{n \choose m}}{ {G(k,\ell;r) \choose m}},
\]
where $m=\min(k,\ell)$.
\end{proof}

We will use the $k=\ell=3$ instance of Lemma \ref{lem4.1new} as stated in the following corollary.

\begin{corollary}\label{th-4-1}
For $r\geq 2$ and $n\geq  G(3,3;r)$, every
$r$-coloring of the edges of $K_n$ contains at least
$\left(\frac{n}{G(3,3;r)}\right)^3$ triangles
that are rainbow or monochromatic.
\end{corollary}

\begin{proof} Let $g = G(3,3;r)$. From Lemma \ref{lem4.1new}, using the trivial bound $s \geq 1$, all that remains to prove is that
$
\frac{{n \choose 3}}{{g \choose 3}} \geq \frac{n^3}{g^3},
$
which holds since $\left(1 -\frac{1}{n}\right)\left(1 - \frac{2}{n}\right) \geq \left(1 -\frac{1}{g}\right)\left(1 - \frac{2}{g}\right)$
holds for $n \geq g$.
\end{proof}

\begin{theorem}\label{thm4.3} Let $r \geq 3$ and define $M(n;r)$ to be the minimum number
of Gallai-Schur triples
over all $r$-colorings of $[1,n]$.
We have
$$
\frac{n^2}{2m^3}(1+o(1)) \leq M(n;r) \leq\frac{4n^2}{121}(1+o(1)) ,
$$
for $n \geq m$, where
$$
m=\begin{cases}
5^{\frac{r}{2}}+1 & \text{ if $r$ is even;}\\
2\cdot 5^{\frac{r-1}{2}}+1 & \text{ otherwise.}
\end{cases}
$$
Moreover, for $r=3$ we have
 {$$ \frac{n^2}{276}(1+o(1)) < M(n;3) \leq \frac{4n^2}{121}(1+o(1)) .$$}

\end{theorem}

\begin{proof}
The upper bound on $M(n;r)$ comes from the following $3$-coloring, and is based on
the coloring found in \cite{RZ98} that produces the minimum number of monochomatic Schur triples over all 2-colorings:
{color every $i \in \left[1,\frac{4n}{22}\right] \cup \left(\frac{10n}{22}, \frac{n}{2}\right]$ red; color every $i \in \left(\frac{4n}{22}, \frac{10n}{22}\right]$
blue; and
color every $i \in \left(\frac{n}{2},n\right]$ green.}
 {From \cite{RZ98},   we know that there are $\frac{n^2}{88}(1+o(1))$ monochromatic Schur triples that occur in  $\left[1,\frac{n}{2}\right]$ and no monochromatic Schur triples that occur in the   $\left(\frac{n}{2},n\right]$.
It is easy to determine that there are $\frac{21n^2}{968}(1+o(1))$ rainbow triples in the set $[1,n]$; see Figure \ref{Fig01}. Hence, in
total there are $\frac{4n^2}{121}(1+o(1))$ Gallai-Schur triples in our 3-coloring.}

\begin{figure}[!h]
\begin{center}
 \hspace*{-20pt}
\setlength{\unitlength}{.35mm}
\begin{picture}(100,130)(10,-10)
\linethickness{.3mm}
\put(0,0){\line(0,1){120}}
\put(0,0){\line(1,0){120}}
\put(120,-10){\line(-1,1){120}}
\put(65,-10){\line(-1,1){65}}
\linethickness{.2mm}

\dottedline{1}(0,20)(110,20)
\dottedline{1}(55,0)(55,110)
\dottedline{1}(0,50)(110,50)
\dottedline{1}(50,0)(50,110)
\dottedline{1}(20,0)(20,110)
\put(108,-7){\footnotesize $n$}
\put(-5,-3){\footnotesize $0$}
\put(-2,125){\footnotesize $y$}
\put(125,-2){\footnotesize $x$}
\put(-7,108){\footnotesize $n$}
\put(-14,18){\footnotesize $\frac{4n}{22}$}
\put(41,-10){\footnotesize $\frac{10n}{22}$}
\put(54,-10){\footnotesize $\frac{n}{2}$}
\put(15,-10){\footnotesize $\frac{4n}{22}$}
\put(-16,48){\footnotesize $\frac{10n}{22}$}
\put(62,-20){\footnotesize $x+y=\frac{n}{2}$}
\put(115,-20){\footnotesize $x+y=n$}

\multiput(7,48)(2,0){6}{,}
\multiput(11,45)(2,0){4}{,}
\multiput(13,42)(2,0){3}{,}
\multiput(17,39)(2,0){1}{,}
\multiput(50,48)(2,0){2}{,}
\multiput(50,45)(2,0){2}{,}
\multiput(50,42)(2,0){2}{,}
\multiput(50,39)(2,0){2}{,}
\multiput(50,36)(2,0){2}{,}
\multiput(50,33)(2,0){2}{,}
\multiput(50,30)(2,0){2}{,}
\multiput(50,27)(2,0){2}{,}
\multiput(50,24)(2,0){2}{,}
\end{picture}
\end{center}
\vskip8pt
\caption{Rainbow solutions (shaded) to $x+y=z$}\label{Fig01}
\end{figure}
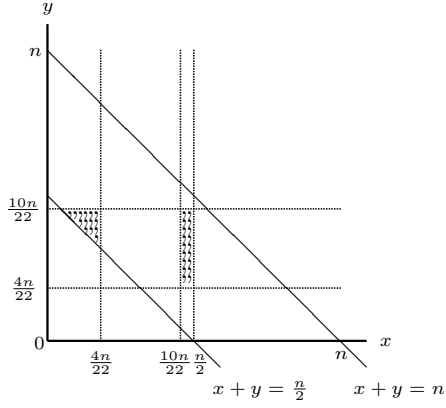

If we require the $r$-coloring to be exact, from this 3-coloring
replace $r-3$ of the integers' colors with $r-3$ distinct colors.  We now have an exact $r$-coloring
of $[1,n]$ with {$\frac{4n^2}{121}- O(rn) = \frac{4n^2}{121}(1+o_r(1))$} Gallai-Schur triples.
(Of course, this is unsatisfying and we address this issue later in the article.)

For the lower bound, we turn to
Theorem \ref{th2-1A}, where we see that $m = G(3,3;r)$.
 {Let $V(K_{n+1})=[1,n+1]$ and let $\chi$ be an $r$-coloring of $[1,n]$. We
define an $r$-coloring $\chi'$ of the edges of $K_{n+1}$ by}
$\chi'(ij)=\chi(|i-j|)$ for any $i,j\in [1,n+1]$. Since $n+1\geq
m$, by Corollary \ref{th-4-1}, we see
that the number of monochromatic and rainbow triangles is at least
$ \frac{n^3}{m^3} (1+o(1))$. Moreover,
every Gallai-Schur triple $(x,y,z)$ under $\chi$
creates at most $2(n+1-z)$ rainbow or monochromatic triangles in $K_{n+1}$ under
$\chi'$.  Using $2(n+1-z) \leq 2n$ gives the bound $\frac{n^2}{2m^3}(1+o(1))$.

However, for specific values of $r$, we can improve this bound slightly.
Let $r$ be given. Noting that there are $\left\lfloor\frac{i}{2}\right\rfloor$ solutions to
$x+y=i$, then at most $\left\lfloor\frac{i}{2}\right\rfloor$ rainbow and monochromatic
solutions to $x+y=i$ exist, each corresponding to at most $2(n+1-i)$ rainbow and
monochromatic triangles.  Letting $k=cn$, where $0<c<1$, we deduce that there
are at least
$$
\sum_{i=2}^{k} \left\lfloor \frac{i}{2}\right\rfloor \cdot 2 (n+1-i) \leq 4 \sum_{i=1}^{k/2} i(n-2i) + O(n)=\left(\frac{k^2}{2}n - \frac{k^3}{3}\right)(1+o(1))
$$
triangles corresponding to
$\frac{k^2}{4}(1+o(1))$ monochromatic and rainbow solutions to $x+y=z$.
{Solving
\begin{equation*}\label{eq1}
\left(\frac{k^2}{2}n - \frac{k^3}{3}\right)(1+o(1)) = \frac{n^3}{m^3}(1+o(1))
\end{equation*}
for $k$ (given $m$) will give a bound of $\frac{k^2}{4}(1+o(1))$ that is slightly better than $\frac{n^2}{2m^3} (1+o(1))$.}

For $r=3$, we can drastically improve the lower bound  (by about a factor of 10) by appealing to a result in \cite{CKPSTY}, where we find that
any $3$-coloring of the edges of $K_n$ admits at least $\frac{n^3}{150}(1+o(1))$ monochromatic triangles.
Hence, we may instead solve
$$
\left(\frac{k^2}{2}n - \frac{k^3}{3}\right)(1+o(1)) = \frac{n^3}{150}(1+o(1))
$$ for $k$ to
obtain (using Maple) that $k \approx 0.1204034549n$ so that we have at least
$$
\frac{k^2}{4}(1+o(1))\approx 0.003624247988n^2(1+o(1)) > \frac{n^2}{276}(1+o(1))
$$
Gallai-Schur triples in any $3$-coloring of $[1,n]$ (actually, we have at least this
many monochromatic Schur triples).
\end{proof}

As mentioned in the above proof, we obtained a lower bound for the minimum number of monochromatic Schur triples
over all $3$-colorings of $[1,n]$. An upper bound
found in \cite{Thot} allows us to state the following.

\begin{corollary}
Let $T(n)$ be the minimum number of monochromatic Schur triples over all $3$-colorings of $[1,n]$.
Then
$$
\frac{1}{276}n^2(1+o(1)) < T(n) < \frac{2.08}{276}n^2(1+o(1)).
$$
\end{corollary}

\section{The Inequality $x+y<z$}\label{Sec5}

In 2010, Kosek, Robertson, Sabo, and Schaal \cite{KRSS10} modified the multiplicity of Schur triples problem of Graham, R\"{o}dl, and Ruci\'{n}ski
\cite{GRR96} by changing $x+y=z$ to the system of inequalities  $x+y<z$ and $x\leq y$, and determined the minimum number of monochromatic solution over all 2-colorings of $[1,n]$.  In this section we approach this problem in the rainbow-monochromatic setting.

\subsection{Minimizing the Number of Monochromatic Solutions to $x+y<z$}

We start this subsection by generalizing the structural result in \cite{KRSS10} from two colors to an arbitrary number of colors
by showing that in order to minimize the number of monochromatic solutions to $x+y<z$ and $x\leq y$, we need only
consider colorings where each color class
consists of a single interval.

To state this structural result, we will use the following notation.

\vskip 5pt
\noindent
{\bf Notation.} Let $n,r \in \mathbb{Z}^+$ and let $0 \leq a_0 \leq a_1 \leq \cdots \leq a_{r-1}$ be integers such that
$\sum_{i=0}^{r-1} a_i = n$.  For {$k \geq 0$} an integer, let
$$
\mathcal{D}_{[k+1,k+n]}(a_0,a_1,\dots,a_{r-1}) = \big\{\chi:[k+1,k+n] \rightarrow \{0,1,\dots,r-1\}$$
\vspace*{-17pt}
$$\hspace*{214pt}
: |\chi^{-1}(i)|=a_i, 0 \leq i \leq r-1\big\};
$$

\vskip 5pt
\noindent
that is, $\mathcal{D}_{[k+1,k+n]}(a_0,a_1,\dots,a_{r-1})$ is the set of all $r$-colorings of $[k+1,k+n]$ where the color $i$ is used exactly $a_i$ times
for $0 \leq i \leq r-1$.  In particular, denote by
$\chi_{[k+1,k+n]}(a_0,a_1,\dots,a_{r-1})  \in \mathcal{D}_{[k+1,k+n]}(a_0,a_1,\dots,a_{r-1})$
the $r$-coloring
$
\chi_{[k+1,k+n]}(a_0,a_1,\dots, a_{r-1})(j)= c
$
where $c\in\{0,1,\dots,r-1\}$ is the unique integer such that $j \in \big[k+1+\sum_{i=0}^{c-1}a_i,$ $ k+\sum_{i=0}^{c}a_i\big]$,
where we take the empty sum to equal 0.

\vskip 5pt
In Theorem \ref{thm5.1} below, we find that the $r$-coloring that minimizes the number of monochromatic solutions to
$x+y < z$ with $x \leq y$ is the one where each color class is a single interval and the interval lengths are ordered in
non-decreasing lengths.

\vskip 5pt
\noindent
{\bf Notation.} We will let $M(\chi)$ represent the number of monochromatic solutions to $x+y<z$ with $x \leq y$ under the coloring $\chi$.

\begin{theorem}\label{thm5.1}
Let integer $k\geq 0$ be fixed and let $ n\in \mathbb{Z}^+$. For non-negative integers $0\leq a_0\leq a_1\leq \cdots\leq a_{r-1}$ (so that $a_0\leq \frac{n}{r}$) and $n=\sum_{i=0}^{r-1} a_i$, we have $M(\chi)\!\geq\! M\!\left(\chi_{[k+1,k+n]}(a_0,a_1,\dots,a_{r-1})\right)$
for any $\chi \!\in\! \mathcal{D}_{[k+1,k+n]}(a_0,a_1,\dots,a_{r-1})$.
\end{theorem}

\begin{proof}[Proof of Theorem \ref{thm5.1}]  We use induction on $r$, with $r=2$ being the result in \cite{KRSS10}.
We assume the result for all $(r-1)$-colorings of $[k+1,k+n]$ for all non-negative integers $k$ and all positive integers $n$.
Consider an arbitrary $r$-coloring of $[k+1,k+n]$ with color class sizes $0 \leq a_0 \leq a_1 \leq \cdots \leq a_{r-1}$.
Identify all integers of colors $r-2$ or $r-1$ with a new color, call it blue, and let $b=a_{r-2}+a_{r-1}$.  We now view the $r$-coloring
as an $(r-1)$-coloring with color class sizes $0 \leq a_0 \leq a_1 \leq \cdots \leq a_{r-3} \leq b$.  By the induction hypothesis,
$\chi_{[k+1,k+n]}(a_0,a_1,\dots,a_{r-3},b)$ achieves the minimum number of monochromatic solutions, with the understanding
that the $b$ ``blue" integers may not produce monochromatic solutions under the original $r$-coloring.

We claim that the $b$ ``blue" integers in $[k+n-b+1 ,k+n]$ when reverted to their original color may be minimized among themselves
(i.e., within the interval $[k+n-b+1 ,k+n]$)
in order to provide the overall minimum number of solutions in our original $r$-coloring.  Assuming otherwise, we must have a ``blue"
integer   interchanged with an integer of color $j \in \{0,1,\dots,r-3\}$.  As argued in \cite{KRSS10}, this cannot
decrease the total number of monochromatic solutions (letting $z \in [k+n-b+1 ,k+n]$ have color $j$ and letting $x < k+n-b+1$ be ``blue" means potentially more
solutions to $x+y<z$ of color $j$ and potentially more ``blue" solutions as well).

Hence, in order to minimize the total number of monochromatic solutions, we should minimize the number of monochromatic
solutions in $[k+n-b+1 ,k+n]$ when the ``blue" integers are reverted to their original colors.  Performing this reversion, we now have a 2-coloring of the
interval $[k+n-b+1 ,k+n]$, and our base case informs us that the first $a_{r-2}$ integers should be one color and
the last $a_{r-1}$ integers should be the other color.  As a result, we obtain the $r$-coloring $\chi_{[k+1,k+n]}(a_0,a_1,\dots,a_{r-1}),$
which completes the induction argument.
\end{proof}

Having the structural result of Theorem \ref{thm5.1}, for any given number of colors, determining the minimum number of solutions
to $x+y<z$ with $x \leq y$ reduces to a straightforward (but  potentially very long) calculus problem.  We present the solution for
$r=3$ and leave any other number of colors to the interested reader.  The main counting lemma we will use is from \cite{KRSS10} and is stated next.

\begin{lemma}[\cite{KRSS10}] \label{CountLemma} Let $I(s,t)$ be the number of solutions to $x+y<z$ with $x \leq y$ that reside in $[s,t]$. Then
$I(s,t) =  \frac{1}{12}(t-2s)^3+O((t-2s)^2) \mbox{for $t > 2s$}$ and is otherwise 0.

\end{lemma}

We now present the $r=3$ optimization result.

\begin{theorem}
For any fixed integer $k\geq 0$, the minimum number of monochromatic solutions to $x+y<z$, $x\leq y<z$ that can occur in any $3$-coloring of $[k+1,k+n]$ is
$M_k(n)=Cn^3+O_k(n^2)$, where $$C=\frac{89-36\sqrt{2}}{63948}\approx 0.0005956138.$$
\end{theorem}

\begin{proof}
Let $I$ be the number of monochromatic solutions to $x+y<z$, $x\leq y$ that can occur in any $3$-coloring of $[k+1,k+n]$.
By Theorem \ref{thm5.1} we need only consider $\chi_{[k+1,k+n]}(a,b,n-a-b)$ with $a \leq b \leq n-a-b$. We may assume that $a\geq k+2$; otherwise,
$[2k+2,k+n]$ is 2-colored and the solution  {is $\frac{n^3}{12(1+2\sqrt{2})^2}+O_k(n^2)$} as given in \cite{KRSS10}. From Lemma \ref{CountLemma},
we need only consider 
 
$$
I=
\begin{cases}  
\frac{(a-k-2)^3+(b-a-k-2)^3+(n-2a-2b-k-2)^3}{12}& \mbox{ if }  a<b-k-2 \mbox{ and}\\ & \hspace*{12pt} a<(n-2b-k-2)/2; \\[2pt]
 \frac{(a-k-2)^3+(b-a-k-2)^3}{12} & \mbox{ if }  a<b-k-2\mbox{ and}\\ & \hspace*{12pt}  a\geq (n-2b-k-2)/2; \\[2pt]
 \frac{(a-k-2)^3+(n-2a-2b-k-2)^3}{12} & \mbox{ if }a\geq b-k-2\mbox{ and} \\ & \hspace*{12pt} a<(n-2b-k-2)/2; \\[2pt]
 \frac{(a-k-2)^3}{12}& \mbox{ if }a\geq b-k-2\mbox{ and}\\ & \hspace*{12pt}  a\geq (n-2b-k-2)/2,
\end{cases}
$$ \normalsize
 {where the expressions are given up to $O_k(n^2)$.}

We use Mathematica to obtain the minimum values over each region.  The result
follows by taking the minimum over these regional minimum values.
\vskip 5pt
For  $a<b-k-2$, $a< (n-2b-k-2)/2$, we obtain
$\frac{89 - 36\sqrt{2}}{63948}n^3+O_k(n^2).$
\vskip 5pt

For $a<b-k-2$, $a\geq (n-2b-k-2)/2$, we obtain
$\frac{9-4\sqrt{2}}{4704}n^3+O_k(n^2).$

\vskip 5pt

For $a\geq b-k-2$, $a<(n-2b-k-2)/2$, we obtain
$\frac{1}{972}n^3+O_k(n^2).$
\vskip 5pt

For $a\geq b-k-2$, $a\geq (n-2b-k-2)/2$, we obtain
$\frac{1}{768}n^3+O_k(n^2).$
\vskip 5pt

The minimum of the above values is $\frac{89 - 36\sqrt{2}}{63948}n^3+O_k(n^2)$, which completes the proof.
\end{proof}

\subsection{Minimizing the Number of Monochromatic and Rainbow Solutions to $x+y<z$}

It is natural to consider the rainbow and monochromatic version of this problem as was done in Theorem \ref{thm4.3}.

\begin{definition}
Let $n,r \in \mathbb{Z}^+$. For a system of linear inequalities $\mathcal{I}$ and an $r$-coloring $\chi$ of $[1,n]$, let
 $G_\chi(\mathcal{I};r)$ represent the number of rainbow    and monochromatic solutions to $\mathcal{I}$ under $\chi$.
 Let $GM(\mathcal{I};r) = \min_\chi (G_\chi(\mathcal{I};r))$, where the minimum is over
 all $r$-colorings of $[1,n]$.
 \end{definition}

Below, we determine an upper bound on $GM(\mathcal{J};3),$ where $\mathcal{J}$ is the system  $x+y<z$ and $x\leq y$.
Our approach will be to limit the colorings investigated, as evidenced by the following notation.

\vskip 5pt
\noindent
{{\bf Notation.} Let $\mathcal{D}$ be the set of $3$-colorings of $[1,n]$ where each color consists of a single interval.
Let $\chi(a,b) \in \mathcal{D}$ be the 3-coloring given by coloring $[1,an]$ red, $(an,(a+b)n]$ blue, and
$((a+b)n,n]$ green, where we have $a,b\geq 0$ and $a+b \leq 1$.}
 Let $GM_{\mathcal{D}}(\mathcal{J};r) = \min_\chi (G_\chi(\mathcal{J};r))$, where the minimum is over
 all $3$-colorings in $\mathcal{D}$.

\vskip 5pt
Clearly  $GM(\mathcal{J};3) \leq GM_{\mathcal{D}}(\mathcal{J};3),$ so we continue
 by determining $GM_{\mathcal{D}}(\mathcal{J};3)$.
From Lemma \ref{CountLemma}, we see that there are $\frac{(d-2c)^3}{12}n^3 + O(n^2)$ solutions to $\mathcal{J}$ in $[cn,dn]$ provided
$d>2c$ (and no solution otherwise).  This allows us to the state the following result.

\begin{lemma}\label{MTrips} Let $\chi(a,b) \in \mathcal{D}$ and let $M(\chi)$ be the number of monochromatic triples  {$(x,y,z)$} such that $x+y<z$ {and $x \leq y<z$},
under $\chi$.  Then, up to $O(n^2)$, we have
$$
M(\chi) = \left\{
\begin{array}{rl}
\frac{1}{12}(a^3+(b-a)^3)n^3&\mbox{if $a<b$ and $a+b > \frac12$}\\[5pt]
\frac{1}{12}(a^3+(b-a)^3+(1-2(a+b))^3)n^3 &\mbox{if $a<b$ and $a+b \leq \frac12$}\\[5pt]
\frac{1}{12}a^3n^3 &\mbox{if $a\geq b$ and $a+b > \frac12$}\\[5pt]
\frac{1}{12}(a^3+(1-2(a+b))^3)n^3 &\mbox{if $a \geq b$ and $a+b \leq \frac12$}\,.
\end{array}
\right.
$$
\end{lemma}

We next enumerate the number of rainbow solutions in a given $\chi(a,b)$.
 {We start by noting that when $x \in [1,an]$ and $y \in (an,(a+b)n]$ with
$(a+b)n \leq x+y \leq n$ we require several distinguishing
cases depending on the sizes of $a$ and $b$; see Figure \ref{RainbowTypes}}.

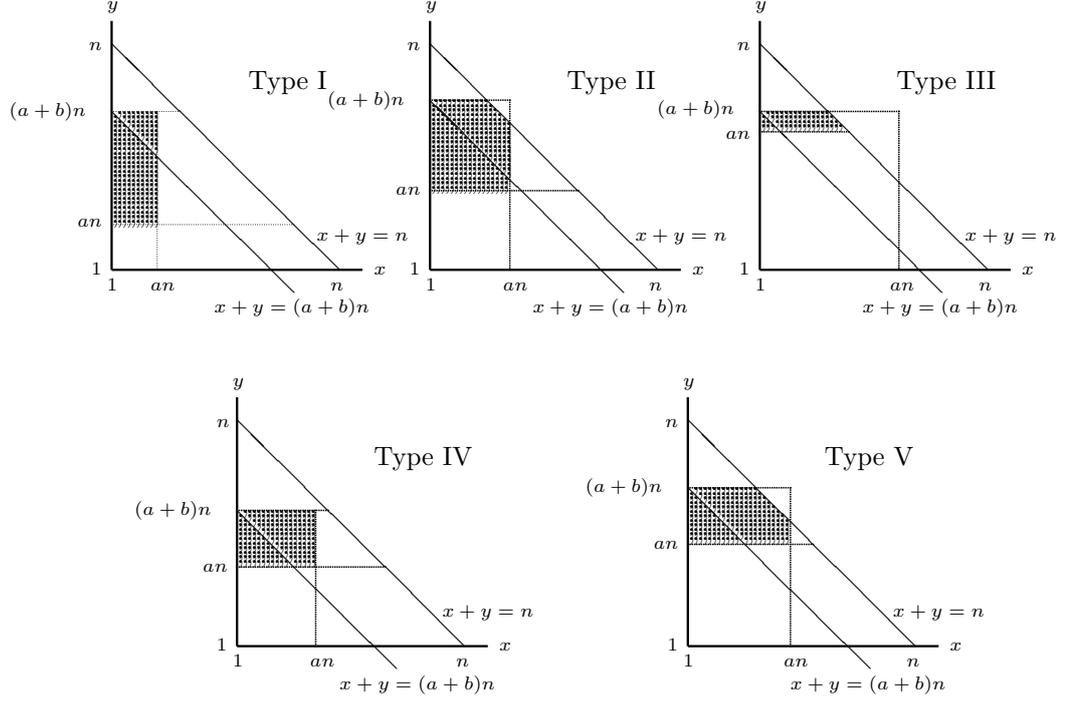
\begin{figure}[h!] 
\setlength{\unitlength}{.3mm}
\begin{subfigure}[t]{.3\textwidth}\setlength{\tabcolsep}{.5em} 
\begin{picture}(0,130)(-45,-10)
\linethickness{.3mm}
\put(0,0){\line(0,1){110}}
\put(0,0){\line(1,0){110}}
\put(100,0){\line(-1,1){100}}
\put(80,-10){\line(-1,1){80}}

\linethickness{.1mm}

\dottedline{1}(0,70)(30,70)
\dottedline{1}(20,0)(20,70)
\dottedline{1}(0,20)(80,20)

\put(96,-9){\footnotesize $n$}
\put(-9,-2){\footnotesize 1}
\put(-2,-9){\footnotesize 1}
\put(-2,115){\footnotesize $y$}
\put(115,-2){\footnotesize $x$}
\put(-10,97){\footnotesize $n$}
\put(-45,68){\footnotesize $(a+b)n$}
\put(17,-9){\footnotesize $an$}
\put(-15,19){\footnotesize $an$}
\put(45,-19){\footnotesize $x+y=(a+b)n$}
\put(90,13){\footnotesize $x+y=n$}

\multiput(2,69)(2,0){9}{,}
\multiput(4,67)(2,0){8}{,}
\multiput(6,65)(2,0){7}{,}
\multiput(8,63)(2,0){6}{,}
\multiput(10,61)(2,0){5}{,}
\multiput(12,59)(2,0){4}{,}
\multiput(14,57)(2,0){3}{,}
\multiput(16,55)(2,0){2}{,}
\multiput(18,53)(2,0){1}{,}

{\multiput(0,67)(2,0){1}{,}
\multiput(0,65)(2,0){2}{,}
\multiput(0,63)(2,0){3}{,}
\multiput(0,61)(2,0){4}{,}
\multiput(0,59)(2,0){5}{,}
\multiput(0,57)(2,0){6}{,}
\multiput(0,55)(2,0){7}{,}
\multiput(0,53)(2,0){8}{,}
\multiput(0,51)(2,0){9}{,}
\multiput(0,49)(2,0){10}{,}
\multiput(0,47)(2,0){10}{,}
\multiput(0,45)(2,0){10}{,}
\multiput(0,43)(2,0){10}{,}
\multiput(0,41)(2,0){10}{,}
\multiput(0,39)(2,0){10}{,}
\multiput(0,37)(2,0){10}{,}
\multiput(0,35)(2,0){10}{,}
\multiput(0,33)(2,0){10}{,}
\multiput(0,31)(2,0){10}{,}
\multiput(0,29)(2,0){10}{,}
\multiput(0,27)(2,0){10}{,}
\multiput(0,25)(2,0){10}{,}
\multiput(0,23)(2,0){10}{,}
\multiput(0,21)(2,0){10}{,}}

\put(60,80){Type I}

\end{picture}
\end{subfigure}
\begin{subfigure}[t]{.3\textwidth}\setlength{\tabcolsep}{.5em}
\begin{picture}(0,130)(-50,-10)
\linethickness{.3mm}
\put(0,0){\line(0,1){110}}
\put(0,0){\line(1,0){110}}
\put(100,0){\line(-1,1){100}}
\put(85,-10){\line(-1,1){85}}

\linethickness{.2mm}

\dottedline{1}(0,75)(35,75)
\dottedline{1}(35,0)(35,75)
\dottedline{1}(0,35)(65,35)

\put(96,-9){\footnotesize $n$}
\put(-9,-2){\footnotesize 1}
\put(-2,-9){\footnotesize 1}
\put(-2,115){\footnotesize $y$}
\put(115,-2){\footnotesize $x$}
\put(-10,97){\footnotesize $n$}
\put(-45,73){\footnotesize $(a+b)n$}
\put(32,-9){\footnotesize $an$}
\put(-15,33){\footnotesize $an$}
\put(45,-19){\footnotesize $x+y=(a+b)n$}
\put(90,13){\footnotesize $x+y=n$}

\multiput(1,74)(2,0){12}{,}
\multiput(3,72)(2,0){12}{,}
\multiput(5,70)(2,0){12}{,}
\multiput(7,68)(2,0){12}{,}
\multiput(9,66)(2,0){12}{,}
\multiput(11,64)(2,0){12}{,}
\multiput(13,62)(2,0){11}{,}
\multiput(15,60)(2,0){10}{,}
\multiput(17,58)(2,0){9}{,}
\multiput(19,56)(2,0){8}{,}
\multiput(21,54)(2,0){7}{,}
\multiput(23,52)(2,0){6}{,}
\multiput(25,50)(2,0){5}{,}
\multiput(27,48)(2,0){4}{,}
\multiput(29,46)(2,0){3}{,}
\multiput(31,44)(2,0){2}{,}
\multiput(33,42)(2,0){1}{,}

{\multiput(0,72)(2,0){1}{,}
\multiput(0,70)(2,0){2}{,}
\multiput(0,68)(2,0){3}{,}
\multiput(0,66)(2,0){4}{,}
\multiput(0,64)(2,0){5}{,}
\multiput(0,62)(2,0){6}{,}
\multiput(0,60)(2,0){7}{,}
\multiput(0,58)(2,0){8}{,}
\multiput(0,56)(2,0){9}{,}
\multiput(0,54)(2,0){10}{,}
\multiput(0,52)(2,0){11}{,}
\multiput(0,50)(2,0){12}{,}
\multiput(0,48)(2,0){13}{,}
\multiput(0,46)(2,0){14}{,}
\multiput(0,44)(2,0){15}{,}
\multiput(0,42)(2,0){16}{,}
\multiput(0,40)(2,0){17}{,}
\multiput(0,38)(2,0){17}{,}
\multiput(0,36)(2,0){17}{,}}

\put(60,80){Type II}
\end{picture}
\end{subfigure}
\begin{subfigure}[t]{.30\textwidth}\setlength{\tabcolsep}{.5em}
\begin{picture}(0,130)(-60,-10)
\linethickness{.3mm}
\put(0,0){\line(0,1){110}}
\put(0,0){\line(1,0){110}}
\put(100,0){\line(-1,1){100}}
\put(80,-10){\line(-1,1){80}}

\linethickness{.2mm}

\dottedline{1}(0,70)(61,70)
\dottedline{1}(61,0)(61,70)
\dottedline{1}(0,61)(39,61)

\put(96,-9){\footnotesize $n$}
\put(-9,-2){\footnotesize 1}
\put(-2,-9){\footnotesize 1}
\put(-2,115){\footnotesize $y$}
\put(115,-2){\footnotesize $x$}
\put(-10,97){\footnotesize $n$}
\put(-45,69){\footnotesize $(a+b)n$}
\put(57,-9){\footnotesize $an$}
\put(-15,58){\footnotesize $an$}
\put(45,-19){\footnotesize $x+y=(a+b)n$}
\put(90,13){\footnotesize $x+y=n$}

\multiput(2,69)(2,0){14}{,}
\multiput(4,67)(2,0){14}{,}
\multiput(6,65)(2,0){14}{,}
\multiput(8,63)(2,0){14}{,}
{\multiput(0,67)(2,0){1}{,}
\multiput(0,65)(2,0){2}{,}
\multiput(0,63)(2,0){3}{,}}

\put(60,80){Type III}

\end{picture}
\end{subfigure}

\vspace*{30pt}
\begin{subfigure}[t]{.33\textwidth}\setlength{\tabcolsep}{.5em}
\begin{picture}(0,130)(-100,-10)
\linethickness{.3mm}
\put(0,0){\line(0,1){110}}
\put(0,0){\line(1,0){110}}
\put(100,0){\line(-1,1){100}}
\put(70,-10){\line(-1,1){70}}

\linethickness{.2mm}

\dottedline{1}(0,60)(40,60)
\dottedline{1}(34.5,0)(34.5,60)
\dottedline{1}(0,35)(65,35)

\put(96,-9){\footnotesize $n$}
\put(-9,-2){\footnotesize 1}
\put(-2,-9){\footnotesize 1}
\put(-2,115){\footnotesize $y$}
\put(115,-2){\footnotesize $x$}
\put(-9,97){\footnotesize $n$}
\put(-45,58){\footnotesize $(a+b)n$}
\put(32,-9){\footnotesize $an$}
\put(-15,33){\footnotesize $an$}
\put(45,-19){\footnotesize $x+y=(a+b)n$}
\put(90,13){\footnotesize $x+y=n$}

\multiput(2,59)(2,0){16}{,}
\multiput(4,57)(2,0){15}{,}
\multiput(6,55)(2,0){14}{,}
\multiput(8,53)(2,0){13}{,}
\multiput(10,51)(2,0){12}{,}
\multiput(12,49)(2,0){11}{,}
\multiput(14,47)(2,0){10}{,}
\multiput(16,45)(2,0){9}{,}
\multiput(18,43)(2,0){8}{,}
\multiput(20,41)(2,0){7}{,}
\multiput(22,39)(2,0){6}{,}
\multiput(24,37)(2,0){5}{,}

{\multiput(0,57)(2,0){1}{,}
\multiput(0,55)(2,0){2}{,}
\multiput(0,53)(2,0){3}{,}
\multiput(0,51)(2,0){4}{,}
\multiput(0,49)(2,0){5}{,}
\multiput(0,47)(2,0){6}{,}
\multiput(0,45)(2,0){7}{,}
\multiput(0,43)(2,0){8}{,}
\multiput(0,41)(2,0){9}{,}
\multiput(0,39)(2,0){10}{,}
\multiput(0,37)(2,0){11}{,}}

\put(60,80){Type IV}

\end{picture}
\end{subfigure}
\begin{subfigure}[t]{.4\textwidth}\setlength{\tabcolsep}{.5em} 
\begin{picture}(0,130)(-150,-10)
\linethickness{.3mm}
\put(0,0){\line(0,1){110}}
\put(0,0){\line(1,0){110}}
\put(100,0){\line(-1,1){100}}
\put(80,-10){\line(-1,1){80}}

\linethickness{.2mm}

\dottedline{1}(0,70)(45,70)
\dottedline{1}(45,0)(45,70)
\dottedline{1}(0,45)(55,45)

\put(96,-9){\footnotesize $n$}
\put(-9,-2){\footnotesize 1}
\put(-2,-9){\footnotesize 1}
\put(-2,115){\footnotesize $y$}
\put(115,-2){\footnotesize $x$}
\put(-10,97){\footnotesize $n$}
\put(-45,68){\footnotesize $(a+b)n$}
\put(42,-9){\footnotesize $an$}
\put(-15,43){\footnotesize $an$}
\put(45,-19){\footnotesize $x+y=(a+b)n$}
\put(90,13){\footnotesize $x+y=n$}

\multiput(2,69)(2,0){14}{,}
\multiput(4,67)(2,0){14}{,}
\multiput(6,65)(2,0){14}{,}
\multiput(8,63)(2,0){14}{,}
\multiput(10,61)(2,0){14}{,}
\multiput(12,59)(2,0){14}{,}
\multiput(14,57)(2,0){14}{,}
\multiput(16,55)(2,0){14}{,}
\multiput(18,53)(2,0){13}{,}
\multiput(20,51)(2,0){12}{,}
\multiput(22,49)(2,0){11}{,}
\multiput(24,47)(2,0){10}{,}

{\multiput(0,67)(2,0){1}{,}
\multiput(0,65)(2,0){2}{,}
\multiput(0,63)(2,0){3}{,}
\multiput(0,61)(2,0){4}{,}
\multiput(0,59)(2,0){5}{,}
\multiput(0,57)(2,0){6}{,}
\multiput(0,55)(2,0){7}{,}
\multiput(0,53)(2,0){8}{,}
\multiput(0,51)(2,0){9}{,}
\multiput(0,49)(2,0){10}{,}
\multiput(0,47)(2,0){11}{,}}

\put(60,80){Type V}
\end{picture}
\end{subfigure}
\vskip 20pt
\caption{Rainbow solutions (shaded) over all possible $\chi(a,b)$}
\label{RainbowTypes}%
\end{figure}

For each given type in Figure \ref{RainbowTypes}, the enumeration of rainbow solutions to
$x+y<z$ up to $O(n^2)$ is easy to produce.  We do so in Table 2.

 \newpage
\begin{table}[h!]
$$
\begin{array}{l|l|l} 
\mbox{Type}&\mbox{Number of Rainbow Solutions}&\mbox{Domain}\\\hline
\mbox{I}&\displaystyle \sum_{x=1}^{an} \sum_{y=(a+b)n-x}^{(a+b)n} \hspace*{-10pt}(n-x-y) +\frac{(2b-a)a}{2}(1-a-b)n^3 &
\begin{array}{l}a<b,\\2a+b<1\end{array}\\\hline
\mbox{II}& \begin{array}{l} \displaystyle\sum_{x=1}^{n-(a+b)n} \sum_{y=(a+b)n-x}^{(a+b)n} \hspace*{-10pt}(n-x-y) + \hskip -8pt \sum_{x=n-an-bn}^{an} \sum_{y=(a+b)n-x}^{n-x} \hspace*{-10pt}(n-x-y)\\[15pt] \hspace*{100pt}\displaystyle +\frac{(2b-a)a}{2}(1-a-b)n^3\\[5pt] \end{array} &
\begin{array}{l}a<b,\\ 2a+b\geq 1\end{array}  \\\hline
\mbox{III}&\displaystyle \sum_{y=an}^{(a+b)n} \sum_{x=(a+b)n-y}^{n-y}  \hspace*{-10pt}(n-x-y) +\frac{b^2}{2}(1-a-b)n^3&
\begin{array}{l}a\geq b,\\ 2a>1 \end{array} \\\hline
\mbox{IV}& \displaystyle \sum_{y=an}^{(a+b)n} \sum_{x=(a+b)n-y}^{an} \hspace*{-10pt} (n-x-y) +\frac{b^2}{2}(1-a-b)n^3&
\begin{array}{l}a\geq b,\\ 2a+b \leq 1\end{array}   \\\hline
\mbox{V}& \begin{array}{l}\displaystyle \sum_{y=an}^{n-an} \sum_{x=(a+b)n-y}^{an}  \hspace*{-10pt}(n-x-y)+ \hskip -3pt \sum_{y=n-an}^{(a+b)n} \sum_{x=(a+b)n-y}^{n-y} \hspace*{-10pt} (n-x-y) \\[15pt] \hspace*{100pt}\displaystyle+\frac{b^2}{2}(1-a-b)n^3\\[5pt]
 \end{array}&
 \begin{array}{l}a\geq b,\\ 2a\leq 1\\1 < 2a+b \end{array}\\\hline
\end{array}
$$ 
\vskip 10pt
\centerline{\normalsize{\bf  Table 2}: Enumerated rainbow solutions over all possible $\chi(a,b)$}
\label{RainTab}
\end{table}

Coupling the expressions in Table 2 with the function in Lemma \ref{MTrips},
we can now give the number of monochromatic and rainbow solutions to $x+y<z$ with $x \leq y$ in $[1,n]$
under $\chi(a,b) \in \mathcal{D}$.  Letting $ G(a,b) =  G_{\chi(a,b)}(\mathcal{J};3)$ and suppressing
all $O(n^2)$ terms, we use
Maple to determine the following:

 $$
\frac{12G(a,b)}{n^3} =
 \left\{
 \begin{array}{rl}
 \displaystyle 1 - 6 a + 12 a^2 - 7 a^3 - 6 b + 36 a b-\\42 a^2 b + 12 b^2 -
 30 a b^2 - 10 b^3 &\mbox{for }a \geq b\mbox{ and }  a+b \leq \frac12;\\[5pt]
\displaystyle  1 - 6 a + 12 a^2 - 10 a^3 - 6 b + 36 a b \\- 33 a^2 b+ 12 b^2 -
 39 a b^2 - 7 b^3 &\mbox{for }a < b\mbox{ and }  a+b \leq \frac12;\\[10pt]
\displaystyle a^3 + 12 a b - 18 a^2 b - 6 a b^2 - 2 b^3 &\mbox{for }a \geq b, a+b > \frac12,\\&\mbox{and } 2a+b \leq 1;\\[5pt]
\displaystyle -2 a^3 + 12 a b - 9 a^2 b - 15 a b^2 + b^3 &\mbox{for }a<b, a+b > \frac12,\\&\mbox{and } 2a+b \leq 1;\\[5pt]
\displaystyle   a^3 + 6 b - 12 a b + 6 a^2 b - 6 b^2 + 6 a b^2&\mbox{for }a \geq b, a \geq \frac12, a+b \leq 1;\\[5pt]
\displaystyle  12 a - 24 a^2 + 17 a^3 + 6 b - 12 a b\\ + 6 a^2 b - 6 b^2 + 6 a b^2 -2&\mbox{for }a \geq b, a<\frac12, 2a+b>1;\\[5pt]
\displaystyle 12 a - 24 a^2 + 14 a^3 + 6 b - 12 a b\\ + 15 a^2 b - 6 b^2 -
 3 a b^2 + 3 b^3-2  &\mbox{for }a < b, a+b \leq 1,\\&\mbox{and } 2a+b>1.
 \end{array}
 \right.
 $$\normalsize

{At this stage it becomes a calculus problem, where we find all critical points where $\frac{\partial G}{\partial a} = \frac{\partial G}{\partial b} =0$
and compare against all the boundaries.  Using Maple, we find that the minimum of
$\left(\frac{3}{196} - \frac{\sqrt{2}}{147}\right)n^3+O(n^2)$ occurs along
the  exterior boundary, in particular at $a = \frac{4-\sqrt{2}}{7}$ and $b=0$
(and two other points both leading to the same 2-coloring, just with different colors).
As can be gleaned in Figure \ref{Fig2}, the minimum occurs when we have 2-colorings (and hence no rainbow solutions); this
minimum is the same as that determined in \cite{KRSS10}.

This gives us
$$
GM(\mathcal{J};3) \leq GM_{\mathcal{D}}(\mathcal{J};3) =\left(\frac{3}{196} - \frac{\sqrt{2}}{147}\right)n^3+O(n^2).
$$
\vspace*{-10pt}
\begin{figure}[h] \centering
\epsfig{file=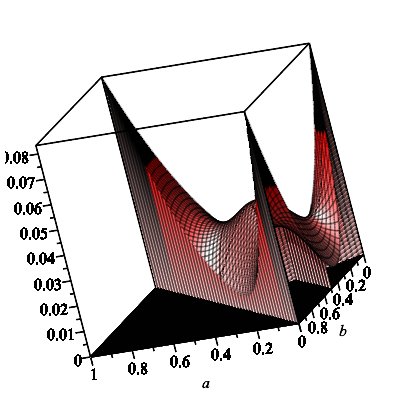, scale=.4} \hspace*{10pt}
\epsfig{file=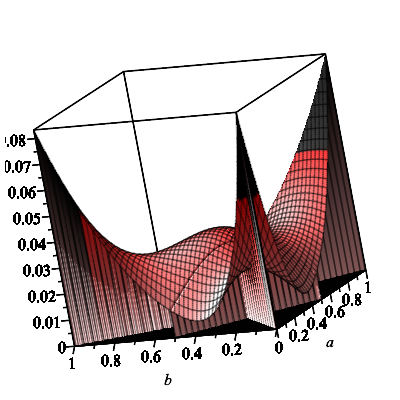, scale=.4}
\vskip 5pt
\caption{Graphs of $\frac{G(a,b)}{n^3}$}\label{Fig2}
\end{figure}

This is a bit of an unsatisfying answer, so we consider a perhaps more appropriate measure of the minimum
monochromatic and rainbow content in $r$-colorings of $[1,n]$.  To this end, consider the following definition.

\begin{definition}\label{def:kexact} We say that an $r$-coloring of $[1,n]$ is {\it $k$-exact} if each color is used at least
$k$ times.
\end{definition}

Before continuing, we should point out that even though the minimum number
of monochromatic and rainbow solutions over all 3-colorings in $\mathcal{D}$
occurs in a 2-coloring, this does not mean we do not have fewer monochromatic
and rainbow solutions over some 3-coloring not in $\mathcal{D}$.  This seems
to be a very difficult problem.

Considering only $\delta n$-exact $3$-colorings in $\mathcal{D}$ (with $\delta \in \left(0,\frac{1}{3}\right]$),
and denoting these by $\mathcal{D}_\delta$, using the work previously
done in this section, we need only adjust the exterior boundary of the problem and
compare all new boundary points to any critical points.    As can be visually seen in Figure \ref{Fig2},
and easily confirmed, no critical point is a local minimum.
Thus, we can conclude that for a given $\delta \in \left[0,\frac13\right)$, the value
of $GM_{\mathcal{D}_\delta}(\mathcal{J};3)$ occurs along one of the planes
$a=\delta$, $b=\delta$, or $a+b=1-\delta$.

\vskip 5pt
\noindent
{\bf Example.} Let $\delta=.1$.  Then $GM_{\mathcal{D}_\delta}(x+y<z;3) = Cn^3 + O(n^2)$,
where
$$
C=\left(\frac{5501}{294000} - \frac{23\sqrt{46}}{147000}\right) \approx 0.01764970347.
$$
This occurs along the boundary where $a=\frac{26-\sqrt{46}}{70}$ and $b=\delta$.

 \subsection{Optimizing the Number of Rainbow Solutions to $x+y<z$}

We have in place all the tools (see Table 2)  to determine the minimum number of rainbow solutions
over all colorings in $\mathcal{D}_\delta$. Obviously, the answer is $O(n^2)$ if
we do not restrict to $\delta n$-exact 3-colorings for some $\delta>0$.  We start by
graphing the function given in Table 2.  The result is in Figure \ref{Fig3}.

\vspace*{-10pt}
\begin{figure}[h] \centering
\epsfig{file=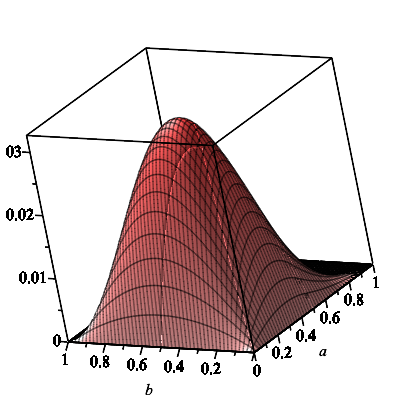, scale=.5}
\vskip 5pt
\caption{Graph of the number of rainbow solutions to $x+y<z$ in colorings $G(a,b)$}\label{Fig3}
\end{figure}

As we can see, asking for the minimum is not the correct question.  This makes some intuitive sense
since rainbow (sometimes referred to as anti-Ramsey) solutions are in some sense a dual of monochromatic (sometimes
referred to as Ramsey) solutions.  Hence, asking for
the maximal number of rainbow solutions to $x+y<z$ over $r$-colorings of $[1,n]$ is the correct question.

As we can see from the graph in Figure \ref{Fig3}, this reduces to finding the maximum function
value over all critical points (on the graph it appears there is only a single relative maximum).
Using Maple we verify that there is indeed a single relative maximum, which is the absolute maximum, which allows us
to state the following result.

\begin{theorem}\label{thmlast} The maximal number of rainbow solutions to $x+y<z$ over all 3-colorings of $[1,n]$ with
each color class being a single interval, is $Cn^3 + O(n^2)$, where
$$
C = \frac{3\sqrt{3}-5}{6} \approx 0.0326920707.
$$
This maximum occurs in the coloring $G\left(2-\sqrt{3},   \frac{\sqrt{3}-1}{2}\right) \approx G(.268,.366)$.
\end{theorem}

We note here that the maximum in Theorem \ref{thmlast} is significantly more than the $\frac{3!}{3^3}\cdot \frac{n^3}{12} + O(n^2) = \frac{n^3}{54} + O(n^3)$
expected under a random $3$-coloring of $[1,n]$ (about 76.5\% more).

We have attempted (in vain) to determine the maximum number of rainbow solutions to $x+y<z$ over all 3-colorings of $[1,n]$
but have made little progress.  We can provide the following bounds.

\begin{theorem} Let ${RM}_3(n)$ be the maximum number of rainbow solutions to $x+y<z$ with $x \leq y$ over all
$3$-colorings of $[1,n]$.  Then
$$
\frac{n^3}{31}(1+o(1)) < {RM}_3(n) \leq \frac{n^3}{27}(1+o(1)).
$$
\end{theorem}

\begin{proof} The lower bound follows from Theorem \ref{thmlast} along with the fact that $\frac{1}{31} < \frac{3\sqrt{3}-5}{6}$.
For the upper bound, consider a 3-coloring with $r$ red, $b$ blue, and $n-r-b$ green integers.
To determine an upper bound on our maximum ${RM}_3(n)$, consider the situation where every possible
rainbow triple satisfies $x+y<z$.  There are $rb(n-r-b)$ such triples.  Maximizing this expression over $r$ and $b$
gives the upper bound of the theorem (when $r=b=\frac{n}{3}$).
\end{proof}


\section{Open Questions}

\begin{question}
\begin{minipage}[t]{4.2in}
Comparing Theorems \ref{th-LR04} and   \ref{th3-color}, explain why
we have equality when $b$ is odd, even though the number of colors used is different.
\end{minipage}
\end{question}

\begin{question}
\begin{minipage}[t]{4.2in}
Noting that the upper bound in Theorem \ref{thm4.3} is independent of $r$ and that
the notion of exact coloring may not be appropriate here, for $r \geq 3$, determine bounds on  the minimum number
of Gallai-Schur triples
over all $k$-exact $r$-colorings of $[1,n]$ (see Definition \ref{def:kexact}).  What conclusion can be drawn when
$k=\delta n$ with  {$\delta \in \left(0,\frac1r\right)$}?
\end{minipage}
\end{question}

\begin{question}
\begin{minipage}[t]{4.2in}
Determine, as a function of $\delta$, the asymptotic minimum number of monochromatic and rainbow solutions to $x+y<z$
that can occur over all $3$-colorings in $\mathcal{D}_\delta$.
\end{minipage}
\end{question}

\begin{question}
\begin{minipage}[t]{4.2in}
Determine, asymptotically, the minimum number of rainbow solutions to $x+y=z$ over all $\delta n$-exact $3$-colorings of $[1,n]$.
\end{minipage}
\end{question}

\begin{question}
\begin{minipage}[t]{4.2in}
Determine, asymptotically, the maximum number of rainbow solutions to $x+y<z$ over all   $3$-colorings of $[1,n]$.
We conjecture that the value in Theorem \ref{thmlast} is the correct value.
\end{minipage}
\end{question}

\vskip 20pt
\section*{Acknowledgement} The authors thank Fred Rowley for useful comments on
a preprint of this article.

\vskip 20pt
\section*{Funding}
Yaping Mao is supported by the JSPS KAKENHI (No. 22F20324).
Chenxu Yang is
supported by the National
Science Foundation of China (No. 12061059) and
the Qinghai Key Laboratory of Internet of Things Project
(2017-ZJ-Y21).  The other authors did not receive support from any
organization for the submitted work.

\section*{Data Availability Statement} No datasets were generated for the current study;
however, computer output is reported and was generated by the aforementioned
{\tt GALRAD} Maple program, available at {\tt http://math.colgate.edu/$\sim$aaron}.

\section*{Declarations}

\subsection*{Conflict of Interest} There is no conflict of interest.

\section*{A. Appendix}

In this appendix we provide a proof of Lemma \ref{lem-2}.  In order to do so,
we will rely on the following lemma.

\begin{lemma}\label{lem-2-2}
Let $n,r \in \mathbb{Z}^+$ with $r\geq 2$. If $\chi$ is a
palindromic Gallai-Schur  $r$-coloring of
 $[1,n]$, then $\psi=\langle \chi, r+1,
\chi, r+2, \chi, r+2, \chi\rangle $ is
 a Gallai-Schur $(r+2)$-coloring of
$[1,4n+3]$.
\end{lemma}

\begin{proof}[Proof of Lemma \ref{lem-2-2}.]
 Suppose, for a contradiction,   that there is a Gallai-Schur  triple
$(x,y,z)$ with $x\leq y<z$ under $\psi$. 

\vskip 5pt
\noindent
{\tt Case 1.} $z \leq 3n+2$.  Let $\gamma=\langle \chi, r+1, \chi, r+2, \chi\rangle $ so that $\gamma(i) = \psi(i)$ for $1 \leq i \leq 3n+2$. Suppose, for a contradiction, 
that there is a Gallai-Schur triple $(x,y,z)$ with
$x\leq y<z$ under $\gamma$. By Lemma \ref{lem-1}, we have $z\geq
2n+2$; otherwise $(x,y,z)$ is a Gallai-Schur
triple under $\chi^*$,  {contradicting the fact that $\chi^*$ is
Gallai-Schur coloring.}
If $z=2n+2$,  {since
$\langle \chi,r+1,\chi\rangle $ is a palindromic coloring, it follows that $\gamma(x) =
\gamma(2n+2-x)=\gamma(y)\neq r+2=\gamma(z)$}, which contradicts the assumption that
$(x,y,z)$ is a Gallai-Schur triple.
Thus $z\in [2n+3,3n+2]$ and $x\leq 2n+1$.

 {If $y\in [2n+3,3n+2]$,} then setting $y'=y-n-1$ and
$z'=z-n-1$ gives $\gamma(y')=\gamma(y)$ and $\gamma(z')=\gamma(z)$,  {where $y', z'\leq 2n+1$}.
It follows that $(x,y',z')$ is a Gallai-Schur
triple under $\gamma$, a contradiction. If $y=2n+2$, then it is easy to see that $\gamma(x)
=\gamma(x+2n+2)\neq \gamma(2n+2)$, a contradiction.  Hence $y\leq 2n+1$.

Now let $x'=2n+2-x$ and $z'=z-2n-2$ {so that $x',z'\leq 2n+1$}.
Since $\langle \chi,r+1,\chi\rangle $ is
 {a palindromic Gallai-Schur coloring}, we see that $\gamma(x')=\gamma(x)$ and $\gamma(z')=\gamma(z)$.
Moreover, $x'+z'=z-x=y$ implies that $(x',z',y)$ or $(z',x',y)$ is a Gallai-Schur triple, the final contradiction
that finishes this case.

\vskip 5pt
\noindent
{\tt Case 2.} $z \geq 3n+3$.
If $z=3n+3$, since $(n+1,2n+2,3n+3)$ is neither
monochromatic nor rainbow under $\psi$, {it follows that} $x\neq n+1$. Since
$\chi$ is {a palindromic coloring},  {it follows that} $\psi(x) = \psi(3n+3-x)=\psi(y)\neq
r+2=\psi(z)$, which contradicts the assumption that $(x,y,z)$ is
monochromatic or rainbow. Hence $z\in [3n+4,4n+3]$ and $x\leq 2n+1$.

 {If  $y\in[3n+4,4n+3]$}, then setting $y'=y-n-1$ and
$z'=z-n-1$ yields $\psi(y')=\psi(y)$ and $\psi(z')=\psi(z)$,  {where $y',z'\leq 3n+2$}.
It follows that $(x,y',z')$ is a monochromatic or rainbow Schur
triple under $\psi$, contradicting the fact
 that $\langle \chi, r+1, \chi,
r+2, \chi\rangle $ is  {a Gallai-Schur coloring} (see Case 1).
If $y=3n+3$, then {$x\leq n$ and} it is easy to see
that $\psi(x) =\psi(x+3n+3)\neq \psi(3n+3)$, a contradiction.  Hence
$y\leq 3n+2$. Now let $x'=3n+3-x$ and $z'=z-3n-3$. Clearly, we have
$\psi(z')=\psi(z)$. If $x\neq n+1$, then $\psi(x')=\psi(x)$. Then
$x'+z'=z-x=y$ implies that either $(x',z',y)$ or $(z',x',y)$ is a Gallai-Schur triple, which contradicts the fact
 that $\langle \chi, r+1, \chi,
r+2, \chi\rangle $ is {a Gallai-Schur coloring}. If $x=n+1$,
then
$\psi(z)=\psi(z-n-1)=\psi(y)\neq \psi(x)=r+1$, which contradicts the
assumption that $(x,y,z)$ is
 monochromatic or rainbow. 
\end{proof}

 We can now present our proof of Lemma \ref{lem-2}. As a reminder of the notation used,
if $\chi$ is an $r$-coloring of $[1,n]$, then
$
\chi^{*}= \langle \chi, r+1, \chi \rangle
$
is an $(r+1)$-coloring of $[1,2n+1]$ and
$
\chi^{**}= \langle \chi, r+1, \chi, r+2,
\chi, r+2, \chi, r+1, \chi \rangle
$
is an $(r+2)$-coloring of $[5n+4]$.

\begin{remark}\label{rem2.2}
Note that for all $i\in [1,n]$ and $j\in\{1,2,3,4\}$, we have
$$\chi^{**}(i+(j-1)(n+1))=\chi^{**}(i+j(n+1)).$$
\end{remark}

\begin{proof}[Proof of Lemma \ref{lem-2}.]
Clearly, $\chi^{**}$ is  {a palindromic coloring}. Suppose to the contrary that
there is a Gallai-Schur  triple $(x,y,z)$ with
$x\leq y<z$ under $\chi^{**}$. By Lemma \ref{lem-2-2} we have $z\geq
4n+4$, for otherwise $(x,y,z)$ is a monochromatic or rainbow Schur
triple under $\langle \chi, r+1, \chi,
 r+2, \chi, r+2, \chi\rangle$,  {contradicting the fact that
$\langle \chi, r+1, \chi, r+2, \chi, r+2,
\chi\rangle $ is a Gallai-Schur coloring.} If $z=4n+4$, since both $(n+1,3n+3,4n+4)$ and
$(2n+2,2n+2,4n+4)$ are neither monochromatic nor rainbow, we have $x\neq
n+1, 2n+2$. But then $\chi^{**}(x) =
\chi^{**}(4n+4-x)=\chi^{**}(y)\neq r+1 =\chi^{**}(z)$, a contradiction. Thus $z\in
[4n+5,5n+4]$.

 {If $y\in[4n+5,5n+4]$}, let $y'=y-n-1$ and
$z'=z-n-1$, then we have $\chi^{**}(y')=\chi^{**}(y)$ and
$\chi^{**}(z')=\chi^{**}(z)$. It follows that $(x,y',z')$ is a
monochromatic or rainbow Schur triple under $\chi^{**}$,
contradicting Lemma \ref{lem-2-2}.
If $y=4n+4$, then  {$x\leq n$ and,} by Remark \ref{rem2.2}, we have
$\chi^{**}(x)
 =\chi^{**}(x+4n+4)\neq \chi^{**}(4n+4)$,
a contradiction.
Hence $y\leq 4n+3$. Now let $x'=4n+4-x$ and
$z'=z-4n-4$. Clearly, we have
$\chi^{**}(z')=\chi^{**}(z)$.
If $x\neq n+1,2n+2$, then $\chi^{**}(x')=\chi^{**}(x)$. Then
$x'+z'=z-x=y$ implies that $(x',z',y)$ or $(z',x',y)$ is a monochromatic or rainbow
Schur triple, which contradicts Lemma \ref{lem-2-2}. If $x=n+1$, then
$\chi^{**}(z)=\chi^{**}(z-n-1)=\chi^{**}(y)\neq \chi^{**}(x)=r+1$,
which contradicts the assumption that $(x,y,z)$ is monochromatic or
rainbow. If $x=2n+2$, then
$\chi^{**}(z)=\chi^{**}(z-2n-2)=\chi^{**}(y)\neq \chi^{**}(x)=r+2$,
which contradicts the assumption that $(x,y,z)$ is monochromatic or
rainbow. Thus, $\chi^{**}$ is {a Gallai-Schur coloring}.
\end{proof}


\begin{thebibliography}{1}\footnotesize

\bibitem{ACPPRT} R. Ageron, P. Casteras, T. Pellerin, and Y. Portella,
A. Rimmel, and J. Tomasik, New lower bounds for Schur and weak Schur numbers,
{\tt 	arXiv:2112.03175}, preprint.


\bibitem{AxenovichIverson}
M. Axenovich and P. Iverson, Edge-colorings avoiding rainbow and
monochromatic subgraphs, \emph{Discrete Math.} 308(20) (2008),
4710--4723.


\bibitem{Behrend}
F.A. Behrend, On sets of integers which contain no three terms in
arithmetical progression, \emph{Proc. Nat. Acad. Sci. U.S.A.} 32
(1946), 331--332.

\bibitem{Berlekamp}
E. Berlekamp, A construction for partitions avoiding long
arithmetical progression, \emph{Canad. Math. Bull.} 11 (1968),
409--414.

\bibitem{BrownLandmanRobertson}
T.C. Brown, B. Landman, and A. Robertson, Bounds on some van der Waerden
numbers, \emph{J. Combin. Theory, Ser. A} 115 (2008), 1304--1309.


\bibitem{Budden}
M. Budden, Schur numbers involving rainbow colorings, \emph{Ars
Math. Contem.} 18(2) (2020), 281--288.











\bibitem{CameronEdmonds}
K. Cameron and J. Edmonds, Lambda composition, {\em J. Graph Theory}
26(1) (1997), 9--16.




\bibitem{ChungGraham}
F.R.K. Chung and R.L. Graham, Edge-colored complete graphs with
precisely colored subgraphs, \emph{Combinatorica} 3(3-4) (1983),
315--324.

\bibitem{CKPSTY}
J. Cummings, D. Kr\'{a}l, F. Pfender, K. Sperfeld, A. Treglown, and M. Young,
Monochromatic triangles in three-coloured
graphs, {\it J. Combin. Theory, Series B} 103 (2013), 489-503.

\bibitem{Datskovsky03}
B. Datskovsky, On the number of monochromatic Schur triples, \emph{Adv. Appl. Math.} 31 (2003), 193--198.








\bibitem{FMO14}
S. Fujita, C. Magnant, and K. Ozeki, Rainbow generalizations of {R}amsey
theory--a dynamic survey, {\em Theo. Appl. Graphs}, 0(1), 2014.






\bibitem{Gallai}
T. Gallai, Transitiv orientierbare {G}raphen, \emph{Acta Math. Acad.
Sci. Hungar} 18 (1967), 25--66.

\bibitem{GRR96}
R. Graham, V. R\"{o}dl, and A. Ruci\'{n}ski, On Schur properties of random subsets of integers, \emph{J. Number Theory} 61 (1996), 388--408.

\bibitem{GyarfasSarkozySeboSelkow}
A. Gy{\'a}rf{\'a}s,  G. S\'{a}rk\"{o}zy, A. Seb\H{o}, and S. Selkow,
Edge colorings of complete graphs without tricolored triangles, {\em
J. Graph Theory} 64(3) (2010), 233--243.


\bibitem{MR2063371}
A. Gy{\'a}rf{\'a}s and G. Simonyi, Edge colorings of complete graphs
without tricolored triangles, {\em J. Graph Theory} 46(3) (2004),
211--216.































\bibitem{Heule} M.J.H. Heule, Schur number five,
Proceedings of the AAAI Conference on Artificial Intelligence {\bf 32} (2018),  6598-6606.

\bibitem{KRSS10}
W. Kosek, A. Robertson, D. Sabo, and D. Schaal, Multiplicity of
monochromatic solutions to $x+y<z$, \emph{J. Combin. Theory, Ser. A}
117 (2010), 1127--1135.

\bibitem{KS03}
W. Kosek and D. Schaal, A note on disjunctive Rado number, \emph{Adv.
Appl. Math.} 31 (2003), 433--439.

\bibitem{KS01}
W. Kosek and D. Schaal, Rado numbers for the equation $\sum^{m-1}_{i=1}
x_i+c=x_m$, for negative values of $c$, \emph{Adv. Appl. Math.}
27(4) (2001), 805--815.






\bibitem{LR04}
B. M. Landman and A. Robertson, \emph{Ramsey Theory on the Integers},
American Mathematical Society, second edition, Providence, 2014.



\bibitem{MORW}
Y. Mao, K. Ozeki, A. Robertson, and Z. Wang, Arithmetic progressions,
quasi progressions, and Gallai-Ramsey colorings, \emph{J.
Combin. Theory, Ser. A} 193 (2023), 105672.











\bibitem{MR1670625}
S.P. Radziszowski, Small {R}amsey numbers, {\em Electron. J.
Combin.}, 1:Dynamic Survey 1, 30 pp. (electronic), 1994.

\bibitem{FRT} A. Robertson, {\it Fundamentals of Ramsey Theory},
CRC Press, Discrete Mathematics and Its Applications, Boca Raton, Florida, 2021.


\bibitem{RobertsonSchaal}
A. Robertson and D. Schaal, Off-diagonal generalized Schur numbers,
\emph{Adv. Appl. Math.} 26 (2001), 252--257.

\bibitem{RZ98}
A. Robertson and D. Zeilberger, A $2$-coloring of $[1,n]$ can have $(1/22)N^2 + O(N)$ monochromatic Schur triples, but not less!,
\emph{Electron. J. Combin.} 5 (1998), R19.


\bibitem{Ram} F. Ramsey, On a problem of formal logic,
{\it Proceedings London Math. Society} {\bf 30} (1930), 264-286.



\bibitem{Rowley} F. Rowley, A generalised linear Ramsey graph construction,
{\it Australas. J. Combin.} 81 (2020), 245-256.

\bibitem{Schaal1} D. Schaal,
On generalized Schur numbers,
{\it Congr. Numer.} 98 (1993),  178-187.

\bibitem{Schaal2} D. Schaal,
A family of 3-color Rado numbers,
{\it Congr. Numer.} 111 (1995),  150-160.

\bibitem{SchaalZinter} D. Schall and M. Zinter,
Continuous Rado number for the equation $a_1x_1+a_2x_2+\cdots+a_{m-1}x_{m-1}+c=x_m$,
{\it Congr. Numer.} 207 (2011), 97--104.

\bibitem{Schoen99}
T. Schoen, The number of monochromatic Schur triples, \emph{European J. Combin.} 20 (1999), 855--866.


\bibitem{Thot} T. Thanatipanonda,
On the monochromatic Schur Triples type problem,
{\it Electron. J. Combin.} 16 (2009), \#R14.


\bibitem{Waerden}
B.L. van der Waerden, Beweis einer Baudetschen Vermutung,
\emph{Nieuw Arch. Wisk.} 15 (1927), 212--216.



\end{thebibliography}
\end{document}